\documentclass[11pt]{amsart}

\usepackage{amsmath, amsthm, latexsym, amssymb, amsfonts}
\usepackage[pdftex]{graphicx,color} 

\usepackage[colorlinks, bookmarksdepth=3, citecolor=blue,linkcolor=blue]{hyperref}

\usepackage{tikz}

\usepackage{tabularx}
\usepackage{diagbox}

\newenvironment{pf}{\proof[\proofname]}{\endproof}
\theoremstyle{plain}
\newtheorem{Th}{Theorem}[section]

\newtheorem{Cor}[Th]{Corollary}
\newtheorem{Prop}[Th]{Proposition}
\newtheorem{Lemma}[Th]{Lemma}
\numberwithin{equation}{section}
\numberwithin{figure}{section}
\theoremstyle{definition}
\newtheorem{Rem}[Th]{Remark}

\newtheorem{Ex}[Th]{Example}
\newtheorem{Def}[Th]{Definition}
\newtheorem{Prob}[Th]{Problem}

\newcommand{\cal}[1]{\mathcal{#1}}

\newcommand{\N}{\mathbb N}
\newcommand{\Z}{\mathbb Z}
\newcommand{\R}{\mathbb R}
\newcommand{\F}{\mathbb F}

\newcommand{\D}{\Delta}
\newcommand{\cA}{\cal A}

\newcommand{\cC}{\cal C}

\newcommand{\cL}{\cal L}

\newcommand{\cO}{\cal O}

\newcommand{\Sig}{\Sigma}

\newcommand{\cchar}{\operatorname{char}}

\newcommand{\conv}{\operatorname{conv}}

\newcommand{\Vol}{\operatorname{Vol}}

\newcommand{\GL}{\operatorname{GL}}
\newcommand{\AGL}{\operatorname{AGL}}

\newcommand{\Gal}{\operatorname{Gal}}

\newcommand{\w}{\operatorname{w}}

\usepackage{mathtools}

\DeclarePairedDelimiter\floor{\lfloor}{\rfloor}

\newcommand{\rs}[1]{Section~\ref{S:#1}}
\newcommand{\rl}[1]{Lemma~\ref{L:#1}}
\newcommand{\rp}[1]{Proposition~\ref{P:#1}}
\newcommand{\rr}[1]{Remark~\ref{R:#1}}
\newcommand{\rex}[1]{Example~\ref{Ex:#1}}
\newcommand{\re}[1]{(\ref{e:#1})}
\newcommand{\rc}[1]{Corollary~\ref{C:#1}}
\newcommand{\rt}[1] {Theorem~\ref{T:#1}}

\newcommand{\rf}[1]{Figure~\ref{F:#1}}
\newcommand{\rpr}[1]{Problem~\ref{Pr:#1}}

\begin{document}


\title[$\F_q$-zeros of sparse trivariate polynomials and toric 3-fold codes]{$\F_q$-zeros of sparse trivariate polynomials\\ and toric 3-fold codes}
\author[Kyle Meyer]{Kyle Meyer}
\author[Ivan Soprunov]{Ivan Soprunov}
\author[Jenya Soprunova]{Jenya Soprunova}
\address[Kyle Meyer]{Department of Mathematics\\ UC San Diego\\ San Diego, CA USA}
\email{kpmeyer@ucsd.edu}
\address[Jenya Soprunova]{Department of Mathematical Sciences\\ Kent State University\\ Kent, OH USA}
\email{esopruno@kent.edu}
\address[Ivan Soprunov]{Department of Mathematics and Statistics\\ Cleveland State University\\ Cleveland, OH USA}
\email{i.soprunov@csuohio.edu}
\keywords{polynomials over finite fields, sparse polynomials,  toric codes, lattice polytopes, Minkowski length}
\subjclass[2020]{Primary 11T06, 52B10, 52B20; Secondary 11T71, 14M25, 14G50, 52B55}

\date{}

\begin{abstract} 

For a given lattice polytope $P\subset\mathbb{R}^3$, consider the space $\cL_P$  of trivariate polynomials over a finite field $\mathbb{F}_q$, whose Newton polytopes are contained in $P$.
We give upper bounds for the maximum number of $\F_q$-zeros 
of polynomials in $\cL_P$ in terms of the Minkowski length of $P$ and $q$, the size of the field. Consequently, this produces
lower bounds for the minimum distance of toric codes defined by evaluating elements of  $\cL_P$ at the points of the algebraic torus $(\F_q^*)^3$.
Our approach is based on understanding factorizations of polynomials in $\cL_P$ with the largest possible number of non-unit factors. The related combinatorial result that we obtain is a description of Minkowski sums of lattice polytopes contained in $P$ with the largest possible number of non-trivial summands. 
\end{abstract}

\maketitle


\section{Introduction}

Let $X$ be a projective algebraic variety  defined over a finite field $\F_q$.
It is a classical problem to estimate the number $N_{X}$ of $\F_q$-points of $X$ in terms of its geometric invariants. 
For curves, the following estimate is due to Hasse and Weil:
$$|N_{X}-(q+1)|\leq 2gq^\frac{1}{2},$$ 
where $X$ is an irreducible curve with arithmetic genus $g$, \cite{AP, LW}. In particular, for smooth plane curves we have $2g=(d-1)(d-2)$, where $d$ is the degree of the curve. In arbitrary dimension such an estimate is provided by the Lang-Weil bound: For an irreducible projective variety 
$X$ of dimension $n$ and degree $d$ defined over $\F_q$ one has
\begin{equation}\label{e:proj}
|N_{X}-\pi_n|\leq (d-1)(d-2)q^{n-\frac{1}{2}}+Cq^{n-1},
\end{equation}
for some constant $C$ depending on $n$, $d$, and $m$ only, see \cite{LW}. Here $\pi_n$ denotes the number of $\F_q$-points in the 
$n$-dimensional projective space, $\pi_n=\frac{q^{n+1}-1}{q-1}$. 
More recently, Ghorpade and Lachuad proved an effective version of this inequality for complete intersections, \cite{GL}. 
In addition, they obtained analogous inequalities for affine varieties. For example, for an affine hypersurface $X_f\subset \bar\F_q^n$ (so $\dim X_f=n-1$) defined by an absolutely irreducible polynomial
$f\in\F_q[x_1,\dots,x_n]$ of degree $d$ there is an explicit bound
\begin{equation}\label{e:affine}
|N_{X_f}-q^{n-1}|\leq (d-1)(d-2)q^{n-\frac{3}{2}}+12(d+3)^{n+1}q^{n-2},
\end{equation}
see \cite[11.3]{GL}. Note that  $N_{X_f}$, the number of $\F_q$-points in $X_f$, is the same as the
number of $\F_q$-zeros of $f$ in  $\F_q^n$ and $q^{n-1}$ is the number of $\F_q$-points in the $(n-1)$-dimensional
affine space.

In this paper we are dealing with polynomials $f$ with prescribed sets of monomials and arbitrary coefficients. As the number of monomials of $f$ can
be significantly smaller compared to a generic polynomial of a given degree, such polynomials are often called {\it sparse}. 
The theory of sparse polynomials (a.k.a. Newton polytope theory) proposes that many algebraic properties of polynomials are reflected in their Newton polytopes.
Let $f\in\F_q[x_1^{\pm 1},\dots,x_n^{\pm 1}]$ be a Laurent polynomial over $\F_q$.
The set of the exponent vectors of the monomials appearing in $f$ is  the {\it support}  of $f$, denoted by $\cA(f)$, so we can write
\begin{displaymath}
f=\sum_{a \in \cA(f)} c_{a}x^{a}, \text{ where } x^{a}=x_{1}^{a_{1}} \dotsm x_{n}^{a_{n}},\ c_{a} \in \F_q^*.
\end{displaymath} 
We put $\cA(f)=\varnothing$ when $f=0$.
The {\it Newton polytope} $P_f$  is the convex hull of the support of $f$.
For any two Laurent polynomials $f,g$ we have $P_{fg}=P_f+P_g$, so one can view the Newton polytope as a natural generalization of the degree of a polynomial.
The sum here is the {\it Minkowski sum} of the polytopes, which is the set of all sums $v_1+v_2$ for all
pairs $v_1\in P_f$ and $v_2\in P_g$. Therefore, factorizations of a sparse polynomial correspond to Minkowski sum decompositions of its Newton polytope. 

\begin{Ex}\label{Ex:one}
Consider $f=1-x+z-x^ay^bz$ for some relatively prime $a,b\in\Z$. Then $P_f$ is a lattice simplex with vertex set
$\{0,e_1,e_3,ae_1+be_2+e_3\}$. Note that if $P_f$ is the Minkowski sum of two lattice polytopes then one of them must be a single point. 
This means that if $f=hg$ then either $h$ or $g$ is a monomial, i.e. a unit in $\F_q[x^{\pm 1},y^{\pm 1},z^{\pm 1}]$. Therefore, $f$ is absolutely irreducible. For the same reason {\it every} non-zero polynomial of the form $f=c_0+c_1x+c_2z+c_3x^ay^bz$ for $c_i\in\F_q$ is either a unit or
absolutely irreducible. 
\end{Ex}

When dealing with sparse polynomials $f$ it is natural to look for $\F_q$-zeros of $f$  none of whose coordinate is zero, i.e., those contained in the {\it algebraic torus} $(\F_q^*)^n$. Throughout the paper we use $N_f$ to denote
this number, i.e.
$$N_f=|\{p\in(\F_q^*)^n : f(p)=0\}|.$$
We will next illustrate how $N_f$ depends on  geometric invariants of $P_f$ (rather than the degree of $f$) such as 
its volume. As it is common in the theory of sparse polynomials, we use $\Vol_n(P)$ to denote the {\it normalized} $n$-dimensional volume of a polytope which is the usual Euclidean volume multiplied by $n!$.

Consider $f=1-x^ay^b\in\F_q[x^{\pm 1},y^{\pm 1}]$ for some relatively prime $a,b\in\N$. This is an irreducible polynomial of degree $a+b$. However, we can make a change of variables $x=u^rv^{-b}$, $y=u^sv^a$ where $r,s\in\Z$ satisfy $ar+bs=1$. This change of variables preserves the number of zeros in  $(\F_q^*)^2$ and turns $f$ into a linear function $f=1-u$ which has $q-1$ zeros in$(\F_q^*)^2$. The geometric reason here is that the Newton polytope $P_f$ is a segment with exactly two lattice points which can be transformed to the unit segment $[0,e_1]$ by an integer linear map. In general, the affine unimodular group $\AGL(n,\Z)$ acts on the Newton polytope $P_f$ and the corresponding monomial changes of variables preserve $N_f$, see \rs{monomial-change}.

\begin{Ex}\label{Ex:two} 
Let $f=1-x+z-x^ay^bz$ as in \rex{one}. To estimate $N_f$ first note that if $x=1$ and $y^b=1$ then $(x,y,z)\in(\F_q^*)^3$ is a zero of $f$ for any $z\in\F_q^*$. If $x\neq 1$ and $x^ay^b\neq 1$ then $(x,y,z)$ is a zero of $f$ for a unique $z\in\F_q^*$. Also, no other $(x,y,z)\in(\F_q^*)^3$ can be a zero of $f$. In the first case we have at most $b(q-1)$ zeros in $(\F_q^*)^3$ and in the second case we have at most $(q-1)^2-2(q-1)+b$ zeros by the inclusion-exclusion principle. Therefore, $N_f\leq (q-1)^2+(b-2)q+2=(q-1)^2+(\Vol_3(P_f)-2)q+2$. See \rt{VolumeBound} for a generalization of this bound. 
\end{Ex}

A related question that we are interested in is how to estimate the number of $\F_q$-zeros of polynomials within a family of sparse polynomials
whose Newton polytopes are contained in some fixed lattice polytope $P$. Our motivation comes from minimum distance estimation for
a class of error-correction codes called {\it toric codes}. In addition, this question is related to factorization properties of sections of line bundles on a toric variety defined over finite fields. We explain both connections in \rs{motivation} below. Given a lattice polytope $P$ in $\R^n$, denote by $\cL_P$ 
the family of sparse polynomials whose Newton polytope is contained in $P$:
\begin{equation}\label{e:L-def}
\cL_P=\{f\in\F_q[x_1^{\pm 1},\dots,x_n^{\pm 1}] : P_f\subseteq P\}.
\end{equation}

\begin{Prob}\label{Pr:all} 
Give a (sharp) upper bound on the number of $\F_q$-zeros in $(\F_q^*)^n$ of non-zero polynomials $f\in\cL_P$ in terms of geometric invariants of $P$.
\end{Prob}

 For example, let $d\D^n$ be the $n$-dimensional simplex with vertex set $\{0,de_1,\dots, de_n\}$ for some positive integer $d< q$. Then 
 $\cL_{d\D^n}$, which we denote  simply by $\cL_d$,  
  is the space of all $n$-variate polynomials of degree at most $d$. It is not hard to see that 
$$\max\{N_f : 0\neq f\in\cL_d\}= d(q-1)^{n-1}.$$
Indeed, it is clear that for distinct $\alpha_1,\dots, \alpha_d$ in $\F_q^*$ the polynomial $f=(x_1-\alpha_1)\cdots(x_1-\alpha_d)$ lies in $\cL_d$ and has exactly $d(q-1)^{n-1}$ zeros in $(\F_q^*)^n$, which is the largest among all polynomials with $d$ linear factors. 
One can then use induction on $n$ to show that any polynomial of degree at most $d$ satisfies the above bound (see \rp{simplex}). This observation motivates a potentially easier problem.
\begin{Prob}\label{Pr:max}  
Give a (sharp) upper bound,  in terms of geometric invariants of $P$, on the number of $\F_q$-zeros in $(\F_q^*)^n$ of polynomials $f\in\cL_P$ that have the largest possible number of non-unit factors.
\end{Prob}

We remark that a solution to \rpr{max} provides a solution to \rpr{all} when $q$ is large enough. Indeed, let 
$f=f_1\cdots f_k$ be a factorization of $f\in\cL_P$. Applying \re{proj} to $X_{f_i}$ and their intersections it is not hard to show that
$|X_{f}(\F_q)|=kq^{n-1}+o(q^{n-1})$. Therefore, for large enough $q$, $|X_{f}(\F_q)|$ is maximal when $k$ is. 

Now let $f\in\cL_P$ be a polynomial with the largest possible number $L$ of non-unit factors and let 
$f=f_1\cdots f_L$. Note that each $f_i$ must be absolutely irreducible. Moreover, it will stay absolutely irreducible no matter how we change its coefficients (as long as at least two of the coefficients are non-zero), otherwise we would have obtained a polynomial in $\cL_P$ with more than $L$ non-unit factors. Continuing our analogy between degrees and Newton polytopes,
each $f_i$ is a ``linear factor" whose irreducibility is due to its ``degree" (i.e. its Newton polytope) rather than its coefficients. The difficulty here is that
the Newton polytope of $f_i$ can be arbitrarily large in volume and  $f_i$ can have many more $\F_q$-zeros than a linear polynomial, as we saw in Examples~\ref{Ex:one} and~\ref{Ex:two}. Although this may not be an issue for special classes of polytopes, e.g. for $P=d\D^n$ considered above,
for arbitrary polytopes this difficulty appears in any dimension $n\geq 3$.
When $n=2$ the situation is more manageable as described in \cite{SoSo1}, which gave a solution to both Problems~\ref{Pr:max} and \ref{Pr:all} for $n=2$, i.e. for bivariate sparse polynomials.
We next present our approach to addressing the aforementioned difficulty  for $n=3$.

\subsection{Our approach and results}\label{S:approach}
As  noted above, factorizations $f=f_1\cdots f_k$ for $f\in\cL_P$ correspond to Minkowski sums of the corresponding Newton polytopes. 
Thus, the geometric invariant that we need to consider in \rpr{max} is the largest number $L$ such that there exist lattice polytopes of positive dimension
$P_1,\dots, P_L$ satisfying $P_1+\dots+P_L\subseteq P$. This invariant is called the {\it Minkowski length} of $P$ and is denoted $L(P)$. 
For example, $L(d\D^n)=d$ (see \rs{mink-length}). In general, there are no simple formulas for $L(P)$, but there exists a polynomial time algorithm for computing $P$ in dimensions 2 and~3, see \cite{BGSW,SoSo1}. 
 
In the current paper we concentrate on the case $n=3$, i.e. the case of trivariate sparse polynomials. 
We next formulate our answer to \rpr{max} for $n=3$.
We write $[0,q-2]^n$ for the $n$-dimensional coordinate cube of side length $q-2$. 

\begin{Th}[\rc{max}]\label{T:intro-main} 
Let $P\subseteq[0,q-2]^3$ be a lattice polytope of Minkowski length~$L$. 
Assume $\cchar(\F_q)>41$ and $q\geq (c+\sqrt{c^2+1})^2$, where $c=\frac{1}{8}\left(\Vol_3(P)-3L+3\right)$. 
Then $$N_f\leq L\,(q-1)^2+2(q-1)(\floor{2\sqrt{q}}-1)$$
for any $f\in\cL_P$ with the largest possible number of absolutely irreducible factors.
\end{Th}

\begin{center}
  \begin{table}
  \renewcommand{\arraystretch}{1.2}
  \begin{tabular}{|c|c|c|c|}
    \hline
    $P$ & vertices & $\Vol_3(P)$ & property \\ 
    \hline
    $T_0$ & $\{e_1,\, e_2,\, -e_1-e_2\}$ & 0 & 2-dim \\ 
    $S_1$ & $\{0,e_1,\, e_2,\, e_3\}$ & 1 & width one \\ 
    $S_2$ & $\{e_1,\, e_2,\, e_3,\, e_1+e_2+e_3\}$  & 2 & width one\\ 
    $E$ & $\{0,e_1,\, e_2,\, e_3,\, e_1+e_2+e_3\}$ & 3 &  width one\\ 
    $K_1$ & $\{e_1,\, e_2,\, e_3,\, -e_1-e_2-e_3\}$  & 4 & Fano \\ 
    $K_2$  & $\{e_1,\, e_2,\, e_1+e_2+2e_3,\, -e_1-e_2-e_3\}$  & 5 & Fano \\ 
\hline
    \end{tabular}
    \vspace{.2cm}
\caption{Possible summands with 4 or more lattice points in a maximal decomposition}
\label{table}
  \end{table}
\end{center}

Despite the fact that the combinatorics of polytopes is much richer starting with $n=3$, the
above bound turned out to be quite similar to the one for $n=2$, which says that
for any $P\subset\R^2$, the number of zeros in $(\F_q^*)^2$ of any $f\in\cL_P$ 
with the largest possible number of absolutely irreducible factors is at most $L\,(q-1)+\floor{2\sqrt{q}}-1$, see \cite[Prop 2.3]{SoSo1}.
The additional factor of $(q-1)$ is expected due to the dimension change, but, remarkably, the combinatorics only contributed to the
extra factor of 2 and additional assumptions on the size of $q$.

To give a solution to \rpr{all} for $n=3$, we show that the bound in \rt{intro-main} holds for large enough $q\geq\alpha(P)$ where the threshold 
$\alpha(P)$ also depends only on the invariants $L(P)$ and $\Vol_3(P)$, see \rt{all}.

Our approach to solving \rpr{max} for $n=3$ is based on first understanding what maximal Minkowski decompositions $P_1+\dots+P_L\subseteq P$ 
may look like for arbitrary $P\subset\R^3$. It follows from the definition of $L=L(P)$ that in any such maximal decomposition we must
have $L(P_i)=1$, $L(P_i+P_j)=2$, $L(P_i+P_j+P_k)=3$, and so on. 

The classification of lattice polytopes with $L(P)=1$, up to the $\AGL(3,\Z)$-action,
was first obtained by Joshua Whitney in his Ph.D. thesis \cite{Josh}. Independently, Blanco and Santos gave such a classification in \cite{BlancoSantos3} where they call lattice polytopes with $L(P)=1$ { distinct pair-sums polytopes} (dps-polytopes) after \cite{CLR}. There are 108 classes of dps-polytopes of lattice width greater than one and infinite families of classes of lattice polytopes of  lattice width one.
For comparison, the classification of planar lattice polytopes with $L(P)=1$, up to the $\AGL(2,\Z)$-action, consists of only three
classes: segments with exactly two lattice points,  triangles with exactly three lattice points, and  triangles with 
exactly three boundary lattice points and one interior lattice point, see \cite[Th 1.4]{SoSo1}.

 We rely on the Whitney and Blanco--Santos result, as well as
previous classification results due to Howe, White, and Kasprzyk \cite{Kasp, Scarf, White} to
classify pairs and triples of polytopes (under appropriate lattice equivalence) which can appear in a maximal Minkowski decomposition. 
In particular, our classification implies that if a maximal Minkowski decomposition contains at least two
polytopes with 4 or more lattice points then each such polytope is equivalent to one in Table~\ref{table}. See also \rf{zoo}
for a visualization.
As a consequence, if a maximal Minkowski decomposition contains more than one  
3-dimensional polytope then the volume of each summand is bounded. 
(If a maximal Minkowski decomposition contains only one 3-dimensional polytope
then its volume can be arbitrarily large as seen in \rex{two}.) Another implication of our results is that in any maximal Minkowski decomposition there could be at most one lattice polytope with 6 or more lattice points, in which case all the other summands are segments. This classification constitutes a substantial combinatorial part of our work and can be of independent interest to those working in discrete geometry and lattice polytopes theory.

\begin{figure}
\begin{center}
\includegraphics[scale=.35]{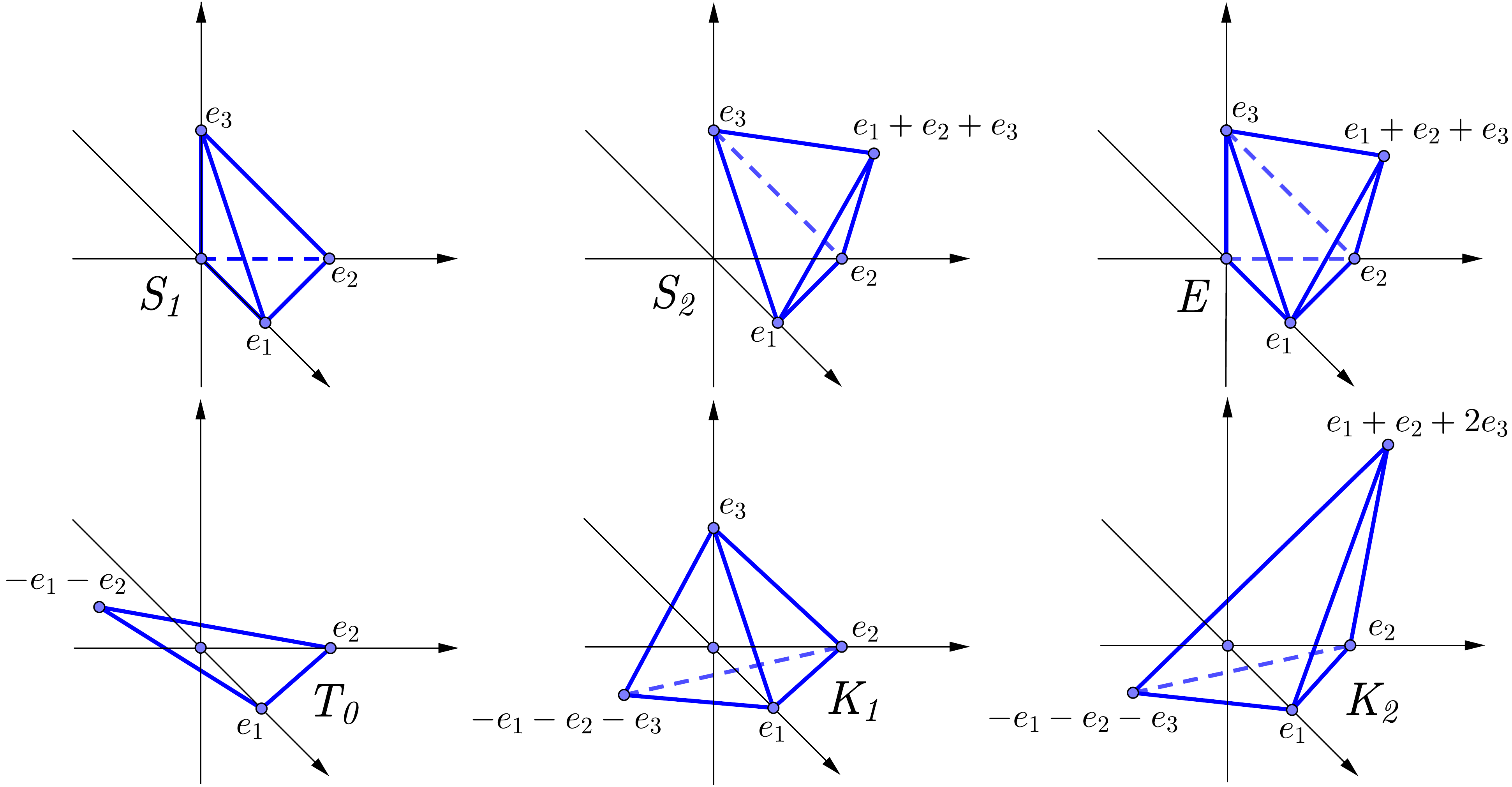}
\end{center}
\caption{Possible summands with 4 or more lattice points in a maximal decomposition}
\label{F:zoo}
\end{figure}

In our second step we estimate $N_f$ in the case when $L(P_f)=1$. Here we rely on Whitney's result \cite[Th 4.30]{Josh} which implies that
for  large enough characteristic of $\F_q$, one has $N_f\leq (q-1)^2+(\Vol_3(P_f)-2)q+2$ when $P_f$ belongs to the 108 classes mentioned above. To prove his result he uses the Grothendieck--Lefschetz trace formula and cohomology computation for the hypersurface $X_f$ in a toric variety corresponding to $P_f$. We use  elementary arguments, similar to the ones in \rex{two}, together with the BKK bound (\rs{mix-vol}), to
show that the same bound holds for the infinite families as well, see \rt{VolumeBound}. This, combined with the combinatorial results about
maximal Minkowski decompositions, provides a bound for $N_f$ when $f$ factors into the largest number of  irreducible factors.
\rt{max-a} gives concrete bounds depending on the number of factors with four or more monomials whereas \rc{max} gives a universal bound.

\subsection{Application to toric varieties and toric codes}\label{S:motivation}  The polynomial spaces $\cL_P$ we defined above appear naturally in toric geometry. Consider a complete $n$-dimensional 
toric variety $X_\Sig$  over $\bar\F_q$, the algebraic closure of $\F_q$. It is defined by a rational 
polyhedral fan $\Sig$ whose 1-dimensional cones $\rho_1\dots, \rho_k$ correspond to the torus invariant prime divisors $D_1,\dots, D_k$ on $X$.
Let $D=\sum_{i=1}^ka_iD_i$ be a torus invariant Cartier divisor on $X_\Sig$ and $\cO(D)$ the corresponding line bundle. It defines a rational polytope $$P_D=\{x\in\R^n : \langle x,v_i\rangle\geq -a_i, 1\leq i\leq k\},$$ where $v_i\in\Z^n$ are the primitive generators of the $\rho_i$.
As shown in \cite[Prop 4.3.3, 4.3.8]{CLSch}, the space of global sections $\Gamma(X_\Sig,\cO(D))$
 is spanned (as a $\bar\F_q$-linear space) by the characters corresponding to the lattice points in $P_D$,
 $$\Gamma(X_\Sig,\cO(D))=\sum_{a\in P_D\cap\Z^n}\bar\F_q\chi^a.$$ 
Then the space of sections that are invariant under the Frobenius automorphism $\Phi\in\Gal(\bar\F_q/\F_q)$ is the $\F_q$-span
of the corresponding characters,
 $$\Gamma^\Phi(X_\Sig,\cO(D))=\sum_{a\in P_D\cap\Z^n}\F_q\chi^a.$$ 
 In the case when $\cO(D)$ is generated by global sections, $P_D$ is a lattice polytope, see \cite[Th 6.1.7]{CLSch}, so 
 $\Gamma^\Phi(X_\Sig,\cO(D))$ is identified with the space $\cL_{P_D}$ we defined in \re{L-def}. Thus, 
 Problems~\ref{Pr:all} and~\ref{Pr:max}  above are related to factorizations of $\F_q$-sections of globally generated line bundles on toric 3-folds and 
 their $\F_q$-zeros.

 Algebraic geometry codes (a.k.a. Goppa codes) are linear error-correcting codes constructed by evaluating
 $\F_q$-sections of a divisor on an algebraic curve at a fixed set of $\F_q$-points of the curve.
It was a major breakthrough in information theory when Tsfasman, Vl\u{a}du\c{t}, and Zink used this construction to improve a
previously known bound for parameters of linear codes (the asymptotic Gilbert-Varshamov bound), \cite{TVZ}. Since then algebraic geometry codes
received much attention and grew into a subfield of both applied algebraic geometry and information theory. Inspired by this construction,
Hansen \cite{Ha1} introduced  {\it toric codes} obtained by evaluating $\F_q$-sections of a divisor on a toric surface  at a fixed set of $\F_q$-points of the surface. This prompted a series of papers on toric codes defined for toric surfaces and higher dimensional toric varieties, see for example 
\cite{CanHibi, Ha2, Jo,Kimball, LSc, LSch, Ru, SoSo1, SoSo2, UmVe}. In particular, the toric code construction for $n=2$ produced about a dozen of new champion codes, see
\cite{BrKasp13, BrKasp-seven,Lit11}, that is codes with largest known minimum distance with given length and dimension, as in Grassl's table of best known codes \cite{Gra}.

Because of the connection described above, toric codes can be defined in purely combinatorial terms without the use of toric geometry as follows. Fix a lattice polytope $P\subset\R^n$ and let $\cL_P$ be the corresponding space
of Laurent polynomials as in \re{L-def}. 
Choose an order of the elements of $(\F_q^*)^n=\{p_1,\dots, p_N\}$, where $N=(q-1)^n$. This defines an evaluation map
$$ev:\cL_P\to \F_q^N,\quad ev: f\mapsto \left(f(p_1),\dots,f(p_N)\right).$$
A {\it toric code} $\cC_{P}$ is the image of the evaluation map, $\cC_{P}=ev(\cL_P)$. 
It is known that if the lattice points of $P$ are distinct in $(\Z/(q-1)\Z)^n$ then the evaluation map is injective, and hence, $\dim\cC_P=\dim{\cL_P}=|P\cap\Z^n|$, \cite{Ru}. In particular, this is true if 
 $P\subseteq [0,q-2]^n$. Recall that the {weight} $w(v)$ of a vector $v\in \F_q^N$
is the number of non-zero entries in $v$. The {\it minimum weight} of a code $\cC\subset \F_q^N$ 
(which for linear codes is the same as the minimum Hamming distance) is defined by $$d=\min\{ w(v) : 0\neq v\in  \cC\}.$$ 
This is an important parameter responsible for the reliability of the code: the larger $d$ is the more errors the code can correct. Thus, it is important to be able to efficiently compute $d$ or at least provide good lower bounds for $d$. In general, when $q$ is large, finding the exact value of $d$  becomes a heavy computational problem. It immediately follows from the above definitions that for a toric code $\cC_P$ we have 
$$d=(q-1)^n-\max\{N_f : 0\neq f\in\cL_P \}.$$ 
Therefore, \rpr{all} is equivalent to finding a (sharp) lower bound for the minimum weight of a toric code, in terms of geometric invariants of $P$.
For $n=3$, the result in \rt{all} provides the following.

\begin{Cor}\label{C:toric-app}
Let $P\subseteq [0,q-2]^3$ be a lattice polytope of Minkowski length $L$. Assume $q\geq\alpha(P)$ where $\alpha(P)$ as in \re{alpha}.
Then the minimum weight of the toric code $\cC_P$ satisfies
$$d\geq (q-1)^3-L\,(q-1)^2-2(q-1)(2\sqrt{q}-1).$$
\end{Cor}
 
 This result is a universal lower bound on the minimum distance of any $3$-fold toric code,
 which does not require extra combinatorial information beyond computing $\Vol_3(P)$ and $L(P)$. 
 Asymptotically, $(q-1)^3-L\,(q-1)^2+O(q^{3/2})$ is what one would expect when comparing this bound 
 to the known explicit answers for the minimum distance of a toric code.
In fact, $d=(q-1)^n-L\,(q-1)^{n-1}+O(q^{n-2})$ for all toric codes considered in
\cite{LSch, Ru, SoSo2}. (Although the Minkowski length $L$ is not discussed in these papers, one can easily
compute it by presenting a Minkowski decomposition of length $L$ in $P$ and embedding $P$ in $L\D^n$.)
 
 The downside of the result in \rc{toric-app} for practical purposes is
 that it holds for rather large size and characteristic of $\F_q$. To get a better estimate on the minimum distance for a particular $P$
 one would first compute all maximal decompositions in $P$ and apply one of the bounds in \rt{max-a}, depending on
 how many summands in a maximal decomposition have 4 or more lattice points. The condition on the characteristic may also be relaxed, see
 \rr{char41}. This method already appears in \cite{Josh}. Together with our results on
 the combinatorics of maximal decompositions this may lead to a practical algorithm for the minimum distance estimation of $3$-fold toric codes.
 This is beyond the scope of this paper. 
 
 On the upside, when $P$ has lattice width one, we prove a better bound with a much smaller threshold 
 $q\geq \beta(P)$ (\rt{width-one}) and give an example when the bound is sharp  (\rex{width-one}).  We state the corresponding application to toric codes below.
 
 \begin{Cor}\label{C:toric-app-2}
Let $P\subseteq\R^3$ be a lattice polytope of lattice width one and Minkowski length~$L$. Assume $q\geq\beta(P)$ where $\beta(P)$ as in \re{beta}.
Then the minimum weight of the toric code $\cC_P$ satisfies
$$d\geq (q-1)^3-L\,(q-1)^2-(q-1)(2\sqrt{q}-1).$$
\end{Cor}
 
 Finally, in \rs{examples} we give examples of toric codes defined by polytopes of lattice width greater than one and compare their parameters with
 theoretical bounds. The minimum distance of three of the presented toric codes exceeds the Gilbert-Varshamov bound for linear codes.
 
 

\subsection{Acknowledgments} We are grateful to anonymous referees for their valuable comments which led to significant improvements of the paper.
Work of Meyer and Soprunova was partially supported by NSF Grant DMS-1156798. 
Several of our classification results on maximal Minkowski decomposition require a computer-assisted search. We use
combinatorial arguments to first reduce the question we are dealing with to a finite number of cases and then use the {\sc Magma} algebra system \cite{Magma} to sort out these cases. We wrote a {\sc Magma}  package to assist us with these computations. 
The code is available at \url{https://github.com/isoprou/minkowski-length}.

\section{Preliminaries}

\subsection{Lattice polytopes}\label{S:prelim} We begin with standard notions in convex geometry and lattice polytope theory.
A  polytope is the convex hull of a finite number of points in a Euclidean space. We use $[v_1,\dots,v_k]$  to denote the convex hull 
of $v_1,\dots, v_k\in \R^n$. We will write vectors in $\R^n$ either in the standard basis $\{e_1,\dots, e_n\}$ or as columns of a matrix.
The lattice polytope $\D^n=[0,e_1,\dots,e_n]$ is called the standard $n$-simplex.
The dimension $\dim P$ of a polytope $P\subset\R^n$ is the dimension of the smallest affine subspace containing the polytope. A hyperplane $H\subset\R^n$ is called a supporting hyperplane for a polytope $P$ if $P\cap H\neq\varnothing$ and $P$ is contained in one of the closed half-spaces defined by $H$. In this case the intersection $P\cap H$ is called a {\it face} of $P$. Faces of dimension 0 are called {\it vertices} and faces of dimension $\dim P-1$ are called {\it facets} of $P$.

We say that a convex polytope $P\subset\R^n$ is a {\it lattice polytope} if all vertices of $P$ are {\it lattice points}, that is, they belong to the integer lattice 
$\Z^n$. We say that a lattice polytope $P$ is {\it empty} if its only lattice points are its vertices. We say that it is {\it clean} if its faces are empty  polytopes.  A lattice segment (i.e. 1-dimensional lattice polytope) which is empty is called {\it primitive}. Equivalently, a lattice segment $[v_1,v_2]$ is primitive if and only if its direction vector $v_2-v_1$ is primitive, that is
the coordinates of  $v_2-v_1$  do not have a common divisor larger than one. 
Recall that $\Vol_n(P)$  denotes the {\it normalized $n$-dimensional volume} of a polytope which is the usual Euclidean volume multiplied by $n!$.
When $P$ is a lattice polytope, $\Vol_n(P)$ is a non-negative integer. 

Let $P\subset\R^n$ be a lattice polytope and $v\in\Z^n$ a primitive vector. The non-negative integer 
$${\rm w}_v(P)=\max_{u\in P}\langle u,v\rangle-\min_{u\in P}\langle u,v\rangle$$
is called the {\it lattice width} of $P$ in the direction of $v$. The {\it lattice width} ${\rm w}(P)$ is smallest value of ${\rm w}_v(P)$ over all primitive $v\in\Z^n$.

Let $\GL(n,\Z)$ be the group of {\it unimodular matrices}, that is, $n\times n$ integer matrices with determinant $\pm 1$. 
The automorphism group of $\Z^n$, which we denote by $\AGL(n,\Z)$, consists of compositions of a multiplication by a unimodular matrix and a translation by a lattice vector. The elements of $\AGL(n,\Z)$ are called {\it affine unimodular maps}. The group  $\AGL(n,\Z)$
acts on the set of all lattice polytopes in $\R^n$. We say that two lattice polytopes in $\R^n$ are {\it  equivalent} 
if there is $\varphi\in\AGL(n,\Z)$ that maps one to the other. A lattice polytope equivalent to $\D^2$ is called a {\it unit triangle} and a lattice polytope equivalent to $\D^3$ is called a {\it unit 3-simplex}.
Clearly, the volume, the width, and the number of (interior) lattice points are $\AGL(n,\Z)$-invariants.

We say that two $k$-tuples $(P_1,\dots,P_k)$ and $(P'_1,\dots,P'_k)$ of lattice polytopes  in $\R^n$ are {\it  equivalent} if there exists an affine unimodular map $\varphi\in{\rm AGL}(n,\Z)$ and lattice vectors $v_i\in\Z^n$ such that $P'_i=\varphi(P_i)+v_i$ for $i=1,\dots,k$. Note that the map $\varphi$ is the same for all~$i$, but the lattice translations are individual.

\subsection{Monomial Changes of Variables}\label{S:monomial-change} Let $f\in\F_q[x_1^{\pm 1},\dots,x_n^{\pm 1}]$ be a Laurent polynomial over $\F_q$ and let $P=P_f$ be its Newton polytope. A translation $P+v$ by a lattice vector $v\in\Z^n$ corresponds to multiplying $f$ by the monomial
$x^v$ and does not change the set of zeros of $f$ in the algebraic torus $(\F_q^*)^n$. Let $U\in\GL(n,\Z)$ be a unimodular matrix. Consider
the monomial change of variables $x_i=y^{u_i}$, where $u_i$ is the $i$-th column of $U$, for $1\leq i\leq n$, and let 
$g(y_1,\dots, y_n)=f(y^{u_1},\dots, y^{u_n})$ be the resulting Laurent polynomial in $\F_q[y_1^{\pm 1},\dots,y_n^{\pm 1}]$. Note that
$x^a=y^{Ua}$ for any $a\in\Z^n$, which implies that the Newton polytope of $g$ is the image of $P$ under $U$. Clearly, the  map
$\phi:(\F_q^*)^n\to(\F_q^*)^n$ defined by  $\phi(t)=(t^{u_1},\dots, t^{u_n})$ is invertible and defines an automorphism of $(\F_q^*)^n$. In particular, 
the image of the set of zeros of $f$ under $\phi$ is the set of zeros of $g$. We have, thus, shown that  $N_f$ is invariant under the $\AGL(n,\Z)$-action on the Newton polytope of $f$.

\subsection{Minkowski length}\label{S:mink-length} 
Recall that the {\ Minkowski sum} $P+Q$ of two  polytopes $P, Q$ in $\R^n$ is the set of all vector sums of their points:
$$P+Q=\{p+q : p\in P, q\in Q\}\subset\R^n.
$$
The Minkowski sum of  lattice polytopes is again a  lattice polytope.

The full Minkowski length $L(P)$ of a lattice polytope $P$ was first defined in~\cite{SoSo1}. In subsequent papers studying this invariant~\cite{BGSW, SoSo3} it was referred to as Minkowski length, and we  use this shorter name in the current paper as well.

\begin{Def} Let $P\subset\R^n$ be a lattice polytope. Then the {\it Minkowski length} $L=L(P)$ is the largest number $L$ such that $P$ contains a Minkowski sum $P_1+\cdots+P_L$ of $L$ lattice polytopes $P_i$, each of positive dimension.
Every such Minkowski sum $P_1+\cdots+P_L$ is then called a {\it maximal decomposition} in $P$.
\end{Def}

Here are a few basic properties of the Minkowski length which either follow directly by definition or are contained in \cite{SoSo3}:
\begin{enumerate}
\item $L(P)$ is an $\AGL(n,\Z)$-invariant,
\item $L(P)$ is superadditive: $L(P+Q)\geq L(P)+L(Q)$, 
\item $L(d\D^n)=d$, where $d\D^n$ is the $d$-dilate of the standard $n$-simplex,
\item $L([0,d]^n)=nd$, where $[0,d]^n$  is the $n$-dimensional coordinate cube of side length $d$,
\item if $P_1+\cdots+P_L\subseteq P$ is a maximal decomposition then $L(P_{i_1}+\dots +P_{i_k})=k$ for
any non-empty $\{i_1,\dots, i_k\}\subset\{1,\dots, L\}$.
\end{enumerate}
Also, there is a simple upper bound on the number of lattice points
in $P$ in terms of $L=L(P)$:
\begin{equation}\label{e:eight}
|P\cap \Z^n|\leq (L+1)^n.
\end{equation}
Indeed, if $P$ has more than $(L+1)^n$ lattice points then two of them coincide modulo $(\Z/(L+1)\Z)^n$, which means that $P$
contains a lattice segment of lattice length $L+1$. In particular, lattice polytopes with $L(P)=1$ have a most $2^n$ lattice points and this bound 
is sharp, see \cite[Th 2.5]{So}. 
A polynomial time algorithm for computing $L(P)$ was provided for $P\subset\R^2$  in~\cite{SoSo1}  and  for $P\subset\R^3$ in~\cite{BGSW}.

\begin{Prop} Let $P\subset\R^n$ be a lattice polytope and $L(P)$ its Minkowski length. Then $L(P)$ is the maximal number of absolutely irreducible factors of polynomials $f\in\cL_P$.
\end{Prop}
\begin{pf} Let $f\in\cL_P$ be a Laurent polynomial with the maximal number of absolutely irreducible factors, $f=f_1\cdots f_L$. Then 
$P_f=P_{f_1}+\dots+P_{f_L}$, $P_f$ is contained in $P$, and $\dim P_{f_i}>0$ since $f_i$ are not units. Therefore, $L\leq L(P)$. 

Conversely, let $P_1+\dots+P_{L(P)}$ be a maximal decomposition in $P$. Choose any Laurent polynomial $f_i\in\F_q[x_1^{\pm 1},\dots,x_n^{\pm 1}]$ with $P_{f_i}=P_i$ for $1\leq i\leq L(P)$. Note that each $f_i$ is absolutely irreducible, otherwise $P_{f_i}$ would decompose into a sum of lattice polytopes of positive dimension, which contradicts the maximality of the decomposition  $P_1+\dots+P_{L(P)}$. Therefore, $L(P)\leq L$.
\end{pf}

\subsection{Mixed volume}\label{S:mix-vol}
The mixed volume $V(P_1,\dots, P_n)$
is the unique multilinear (with respect to Minkowski addition) function of $n$-tuples of convex polytopes (more generally compact convex sets) in $\R^n$ which coincides with the volume on the diagonal,
$$V(P,\dots, P)=\Vol_n(P).$$
One of the fundamental results in the theory of Newton polytopes is the Bernstein-Khovanskii-Kushnirenko theorem (a.k.a. the BKK bound) which relates the intersection number of generic hypersurfaces in the algebraic torus and the mixed volume of their Newton polytopes, see \cite{Be, Kho, Kush}.
Basic properties of the mixed volume include invariance with respect to independent translations of the $K_i$
and simultaneous unimodular transformations of the $K_i$, as well as monotonicity with respect to inclusion in each of the $K_i$, see
\cite[Sec 5]{Sch}.
The following  property of the mixed volume  goes back to the work of Minkowski \cite{Min}.
\begin{Prop}\label{P:Mink} 
Let $P_1,\dots, P_n\subset\R^n$ be convex polytopes. Then $V(P_1,\dots, P_n)=0$ if and only if there exists 
non-empty $\{i_1,\dots,i_k\}\subseteq\{1,\dots, n\}$ such that $\dim(P_{i_1}+\dots+P_{i_k})<k$.
\end{Prop}
We also note that when $P_1,\dots, P_n$ are lattice polytopes  $V(P_1,\dots, P_n)$ is a non-negative integer. We will make use
of the following simple observations.

\begin{Lemma}\label{L:mix} Let $P_0, P_1\subset\R^2$ be convex polytopes and $P\subset\R^3$ be 
the convex hull of $(P_0\times\{0\})\cup (P_1\times\{1\})$. Then $\Vol_3(P)=\Vol_2(P_0)+V(P_0,P_1)+\Vol_2(P_1).$
\end{Lemma}

\begin{pf}
We  have
$$\Vol_3(P)=3\int_{0}^1\Vol_2((1-t)P_0+tP_1)\,dt.$$
Using $\Vol_2(A+B)=V(A+B,A+B)$ and the bilinearity of the mixed volume we get $\Vol_2((1-t)P_0+tP_1)=(1-t)^2\Vol_2(P_0)+2t(1-t)V(P_0,P_1)+t^2\Vol_2(P_1)$. Integrating, we obtain the claim.
\end{pf}

\begin{Lemma}\label{L:VolumeBound} 
Let $P_1,P_2\subset \R^3$ be lattice polytopes such that $P_1+P_2$ is $3$-dimensional. Then 
$\Vol_3(P_1+P_2)\geq \Vol_3(P_1)+3$.
\end{Lemma}

\begin{pf} Since $P_1+P_2$ is $3$-dimensional, we can chose a lattice segment $I\subseteq P_2$ such that $P_1+I$ is $3$-dimensional.
Then, by the multilinearity of the mixed volume,
$$\Vol_3(P_1+P_2)\geq \Vol_3(P_1+I)=\Vol_3(P_1)+3V(P_1,P_1,I)+3V(P_1,I,I)+\Vol_3(I).$$
By \rp{Mink} the last two summands are zero and  $V(P_1,P_1,I)\geq 1$, since $P_1+I$ is $3$-dimensional.
\end{pf}

\section{Classifying pairs $(P,Q)$ with $L(P+Q)=2$}\label{S:classification-2}

In this section we concentrate on classifying all pairs $(P,Q)$ of lattice polytopes in $\R^3$, up to equivalence, satisfying $L(P)=L(Q)=1$ and $L(P+Q)=2$.  To simplify notation we write $|P|$ for the number of lattice points of a polytope $P\subset\R^3$, i.e. $|P|=|P\cap\Z^3|$.
We call $|P|$ the {\it size of $P$}.

\subsection{First observations}
In many of our combinatorial arguments below we reduce the problem to a finite number of cases which are then sorted 
with the aid of {\sc Magma}  \cite{Magma}. Throughout this section we mention some of the functions we wrote  for this purpose and explain what they do. 
A complete list of functions with code and description is contained in \url{https://github.com/isoprou/minkowski-length}. 

The proof of the next lemma uses the function {\tt minktwo}, which checks whether $L(P)\leq 2$ for a given lattice polytope $P$.
Our package also includes functions {\tt minkone} and {\tt minkthree}, which check whether $L(P)=1$ and $L(P)\leq 3$, respectively.

\begin{Lemma}\label{L:widthbound}
Let $P$ be a lattice triangle in $\R^3$ that satisfies $L(P)=1$ and let  $I$ be a primitive segment such that  $L(P+I)=2$. Let $v$ be a primitive normal to the lattice plane that contains~$P$.
\begin{itemize}
\item[(1)] If $P$ is a unit triangle then $\w_v(I)\leq 14$.
\item[(2)] If $P$ is equivalent to $T_0$ then $\w_v(I)\leq 2$.
\end{itemize}  
\end{Lemma}

\begin{proof} 
In part (1) we can assume that $P=[0, e_1, e_2]$ is the standard 2-simplex. By applying a matrix  of the form $\left(\begin{matrix}1&0&\ast\\ 0&1&\ast\\0&0&\pm1\end{matrix}\right)$, which fixes the $(x,y)$-plane, we can ensure
that the components of the direction vector $u=(p,q,r)$ of $I$ satisfy $0\leq p, q\leq r$. Due to symmetry we can further assume that  $0\leq p\leq q\leq r$. We then have $v=e_3$ and  $\w_v(I)=r$, so we need to show that $r\leq 14$.

Let $\Pi$ be the half-open parallelepiped  spanned by the vectors $e_1, e_2$, and $u$:
$$\Pi=\{\lambda_1e_1+\lambda_2e_2+\lambda_3u \mid 0\leq \lambda_i<1\}.$$
Note that $r$ is the Euclidean volume of $\Pi$ and, hence, $\Pi$ contains exactly $r$ lattice points (see \cite[VII.2.5]{Bar}).
On the other hand, the plane containing $e_1,e_2, u+e_1$, and $u+e_2$ subdivides the closure of $\Pi$ into $P+I$ and
its equivalent copy. Let $e$ be the number of interior lattice points in $P+I$ and $f$ be the number of lattice points in the relative interior of the facets of $P+I$.  Observe that $P+I$ has no lattice points in the  relative interior of the edges and the triangular facets. Then, on the one hand, $|P+I|=e+f+6$, and on the other hand
$r=|\Pi|=2e+f+1$. This implies $2|P+I|-r\geq 11$. 

Now, since $L(P+I)=2$ there are at most 27 lattice points in $P+I$, see \re{eight}. Combining with the previous inequality, we obtain $r\leq 43$. We next use {\sc Magma}  to check  for each $(p,q,r)$  with $0\leq p,q\leq r\leq 43$ whether $L(P+I)=2$. This allows us to improve the bound to $r\leq 14$ and the 
conclusion follows.

In part (2) we can assume that $P=T_0$. Since $T_0$ contains a unit triangle, by part (1) the direction vector $(p,q,r)$ of $I$ satisfies $0\leq p\leq q\leq r\leq 14$ . Using {\sc Magma}  we check for each such triple $(p,q,r)$  whether $L(P+I)=2$ and improve the bound to $r\leq 2$.
\end{proof}

Let  $P$ be lattice polytope with $L(P)=1$ and let  $I$ be a primitive segment  that satisfies  $L(P+I)=2$.  Let  $u=(x,y,z)$ be a direction vector of $I$. Then by \rl{widthbound}, for each face $F_i$ with a normal vector 
$v_i=(a_i,b_i,c_i)$ we have $|a_ix+b_iy+c_iz|\leq 14$. Our function {\tt GoodPolytope(P)} outputs the polytope $G$ defined by these inequalities. Furthermore, the function {\tt FindSegments(P)}
runs through all primitive $(x,y,z)\in G$, checks whether for $I=[(0,0,0), (x,y,z)]$  we indeed have $L(P+I)=2$, and outputs all $(x,y,z)$ for which this holds true. The function {\tt AddTriangleHuh(P)} checks whether such primitive segments can be used to form a lattice triangle $T$ with $L(P+T)=2$.  The function {\tt FindTriangles(P)} lists all such triangles. The function {\tt AddTetraHuh(P)} checks whether one can use these segments to form a tetrahedron $T$, possibly degenerate, that satisfies $L(P+T)=2$. The function {\tt FindTetra(P)} lists all such tetrahedra.

\vspace{.2cm}

We now turn to our classification problem.  A pair of primitive lattice segments $(P,Q)$ with $L(P+Q)=2$ is equivalent to either $([0,e_1],[0,e_2])$ or $([0,e_1],[0,e_1+2e_2])$, \cite[Lemma 1.7]{SoSo1} (see also the proof of \rl{three_segments}). As we just explained, given a lattice polytope $P$  one can use \rl{widthbound} and function {\tt FindSegments} to list all primitive segments $Q$ that satisfy $L(P+Q)=2$.  In particular, we can enumerate up to the equivalence all  pairs $(P,Q)$ where $P$ is a unit triangle and $Q$ is a primitive segment that satisfy $L(P+Q)=2$.  If $P$ and $Q$ are both unit triangles we can  assume that $P=[0,e_1,e_2]$  and then use the function  {\tt FindTriangles} to list all the unit triangles $Q$ that satisfy $L(P+Q)=2$. We do not write these lists explicitly here as they are long and  not needed for our main results. Thus, in what follows we assume that $|P|\geq 4$ and $|Q|\geq 3$. In addition, we can assume $|P|\geq |Q|$, by symmetry. Recall also that $L(P)=L(Q)=1$ implies $|P|, |Q|\leq 8$, by \re{eight}. All these cases are covered by the following four possibilities:
\begin{enumerate}
\item $P$ is an empty lattice polytope,
\item $P$ is equivalent to $T_0$,
\item $P$ has a subpolytope equivalent to $T_0$,
\item $P$ is clean, but not empty.
\end{enumerate}
We address each of them below.

\subsection{$P$ is an empty lattice polytope}
In this subsection we assume that $P$ is an empty lattice polytope, i.e., the only lattice points of $P$ are its vertices. We treat the cases  $|P|=4$ and $|P|=5$ separately and then show that for  $|P|\geq 6$ 
there is no lattice polytope $Q$ with $|Q|\geq 3$ and $L(P+Q)=2$.

White's theorem ~\cite{White}  states that every empty lattice tetrahedron in $\R^3$ is  equivalent to a tetrahedron 
$$T_{a,b}=\begin{bmatrix}1&0&0&a\\0&1&0&b\\0&0&1&1\end{bmatrix}.
$$
(We use such matrix form to denote the convex hull of the columns of the matrix.)
Here $a$ and $b$ are non-negative, relatively prime integers. Moreover, $T_{a,b}$ is  equivalent to $T_{a',b'}$ if and only if $a+b=a'+b'$ and $a'=\pm a^{-1} {\rm mod}(a+b)$.  It was further shown in~\cite{Scarf} that any empty lattice polytope in $\R^3$ is of width one. 

\begin{Lemma}\label{L:empty_tetrahedra}
Let $P$ be an empty lattice tetrahedron and  let $Q$ be a lattice polytope such that $L(Q)=1$ and $L(P+Q)=2$. Then
\begin{itemize}
\item[(1)] If $|Q|=4$ then $\Vol_3(P)\leq 2$;
\item[(2)] If $|Q|=3$ then $\Vol_3(P)\leq 5$.
\end{itemize}
\end{Lemma}
\begin{proof} 
By White's Theorem we can assume that $P=T_{a,b}$, where $b\geq a\geq 0$, $a+b>0$, and $\gcd(a,b)=1$.
Then $\Vol_3(P)=a+b$ and the normals to the facets of $P$ are 
$$n_1=(1,1,1),\ n_2=(-1,-1,a+b-1),\ n_3=(-b,a,-b),\  {\rm and}\ n_4=(b,-a,-a).$$
They satisfy the following relations
\begin{equation}\label{e:relations}
(a+b)e_1=an_1+n_4,\ (a+b)e_2=bn_1+n_3,\ (a+b)e_3=n_1+n_2.
\end{equation}
Let $u=(x,y,z)$ with $z\geq 0$ be a direction vector of a primitive lattice segment in $Q$. 
By part (1) of Lemma~\ref{L:widthbound} we have $|\langle u,n_i\rangle|\leq 14$ for $1\leq i\leq 4$.
Therefore, by \re{relations},
$$(a+b)|x|=|\langle u,an_1+n_4\rangle|\leq a|\langle u,n_1\rangle|+|\langle u,n_4\rangle|\leq 14(a+1).$$
Similarly, we obtain $(a+b)|y|\leq 14(b+1)$ and $(a+b)|z|\leq 28$.

Suppose first that  $a+b>28$. Then  $z=0$ and, hence, $|x+y|=|\langle u,n_1\rangle|\leq 14$ and
$|bx-ay|=|\langle u,n_3\rangle|\leq 14$. Switching the sign of $u$ if needed, we can assume
 that $x\geq 0$. But if $x\geq 1$ and $y\leq -1$ then $14\geq |bx-ay|=bx+a(-y)\geq a+b>28$, a contradiction. Hence, we can assume that $x,y\geq 0$.
We next use {\sc Magma}  to find all segments $I$ with direction vector $(x,y,0)$ such that  $x+ y\leq 14$, $x,y\geq 0$ and the sum of $I$ and $\conv\{e_1,e_2,e_3\}\subset P$ has Minkowski length~2.
The only such segments are the ones with $(x,y)=(1,0)$, $(0,1)$, and $(1,1)$. Note that for the first two of these options 
the condition $|bx-ay|\leq 14$ implies $a\leq 14$ and $b\leq 14$, which contradicts  $a+b>28$. Hence, 
 these two directions cannot appear simultaneously. We conclude that  there is no $Q$ with $|Q|\geq 3$ in the case when  $a+b>28$.

Next  for each $P$ with $0\leq a\leq b$,  $a+b\leq 28$, and $\gcd(a,b)=1$ we run the functions {\tt AddTriangleHuh} and {\tt AddTetraHuh} and 
conclude that  $a+b\leq 2$ in case (1) and $a+b\leq 5$ in case (2), and the conclusions of the lemma follow.
\end{proof}

\begin{Rem}\label{R:empty+3}
 Note that  by White's theorem there are six equivalence classes of empty tetrahedra with volume at most 5 with $(a,b)=(0,1), (1,1), (1,2), (1,3), (1,4), (2,3)$.
Using the functions {\tt GoodPolytope} and {\tt FindTriangles} one can enumerate (up to equivalence) all $(P,Q)$ with $L(P+Q)=2$, where $P$ is an empty tetrahedron and $Q$ is a unit triangle.
\end{Rem}

\begin{Th}\label{T:two_empty_tetrahedra} Let $P,Q\subset\R^3$ be empty tetrahedra such that $L(P+Q)=2$. Then each of $P$ and $Q$ is equivalent to $S_1$ or $S_2$, as defined in Table~\ref{table}.
\end{Th}
\begin{proof} In the notation introduced in Lemma~\ref{L:empty_tetrahedra} we proved that $a+b\leq 2$, so we have two options,  $a=0, b=1$ and $a=b=1$, where the first option gives a tetrahedron equivalent to $S_1$ (since both have volume 1) and the second one gives $S_2$.
\end{proof}

The next proposition shows that if each of $P$ and $Q$ is equivalent to $S_2$ then 
$P$ and $Q$ coincide up to lattice translation.

\begin{Prop}\label{P:twoS_2}  
Suppose that lattice polytopes $P,Q\subset\R^3$ are each equivalent to $S_2$ and that $L(P+Q)=2$. Then up to a lattice translation we have  $P=Q$.
\end{Prop}
\begin{proof}
We run the function {\tt FindTetra} for $S_2$ which  outputs all $Q$ with $|Q|=4$ such that $L(S_2+Q)=2$. 
We then check which tetrahedra in the output are of normalized volume 2. 
The output contains two tetrahedra, each of which is a lattice translation of $S_2$.
\end{proof}

\begin{Th}\label{T:empty_5points}
Let $P$ be an empty lattice polytope with $|P|=5$ and $L(P)=1$. Let $Q$ be a lattice polytope such that $|Q|=4$, $L(Q)=1$, and $L(P+Q)=2$.
Then the pair $(P, Q)$ is  equivalent to the pair $(E, S_2)$, as defined in Table~\ref{table}. 
\end{Th}
\begin{proof}
By Howe's Theorem~\cite{Scarf}, $P$ is equivalent to $\left[\begin{matrix} 0&1&0&0&a\\ 0&0&1&0&b\\ 0&0&0&1&1\end{matrix}\right]$, where $b\geq a\geq 0$, $a+b>0$, and $\gcd(a,b)=1$.
By removing the origin from the set of vertices of $P$ and applying Lemma~\ref{L:empty_tetrahedra} to the obtained empty tetrahedron we conclude that $a+b\leq 2$, and, as in Theorem~\ref{T:two_empty_tetrahedra},  
either  $a=b=1$ or $a=0$, $b=1$, where the latter is ruled out since then  $L(P)=2$.

For $a=b=1$, that is, for $P=E$ we run the function {\tt FindTetra} which outputs all $Q$ with $|Q|=4$ such that $L(P+Q)=2$.  The output is a lattice shift of $S_2$.
\end{proof}

\begin{Prop}\label{P:empty_6points}
Let $P$ be an empty lattice polytope with $|P|=6$ and $L(P)=1$. Let $Q$ be a lattice polytope such that $|Q|\geq 3$ and  $L(Q)=1$. Then  $L(P+Q)\geq 3$.
\end{Prop}
\begin{proof} 
By Howe's Theorem~\cite{Scarf} any empty lattice 3-polytope is of width one. Size 6 lattice polytopes are classified in Tables 6 and 7 of~\cite{BlancoSantos1}.
All the polytopes in these two tables except for the last entry in Table 7 are either non-empty or have Minkowski length at least 2.
We therefore can assume that $P=\left[\begin{matrix} 0&1&0&0&a&c\\ 0&0&1&0&b&d\\ 0&0&0&1&1&1\end{matrix}\right]$, where $a,b,c,d> 0$, $c+d>a+b$, and $ad-bc=\pm 1$.

It follows from Lemma~\ref{L:empty_tetrahedra} applied to the subpolytopes $\left[\begin{matrix} 1&0&0&a\\ 0&1&0&b\\ 0&0&1&1\end{matrix}\right]$ and  $\left[\begin{matrix} 1&0&0&c\\ 0&1&0&d\\ 0&0&1&1\end{matrix}\right]$
that $a+b\leq 5$ and $c+d\leq 5$. For each such $P$  we run function {\tt AddTriangleHuh} and conclude that there is no $Q$ with $|Q|=3$ and $L(P+Q)=2$.
\end{proof}

\subsection{$P$ is equivalent to $T_0$} Now we consider the case when $P$ is equivalent to $T_0$, 
i.e., $P$ is a lattice triangle with three boundary lattice points and one interior lattice point. Note that the subgroup of symmetries of $T_0$ in ${\rm GL(2,\Z)}$ is isomorphic to the symmetric group $\mathbb{S}_3$ and is generated by 
$\begin{bmatrix}0&-1\\1&-1\end{bmatrix}$ and $\begin{bmatrix}0&1\\1&0\end{bmatrix}$. 

\begin{Prop}\label{L:T_0plusQ}  Let $Q\subset\R^3$ be a lattice triangle with  $|Q|=3$ such that  $L(T_0+Q)=2$. Then, after a lattice translation and  a unimodular transformation that maps $T_0$ to itself, $Q$ is equal to
$$\begin{bmatrix}0&1&0\\0&0&0\\0&0&1\end{bmatrix} \text{ or}\ 
\begin{bmatrix}0&0&a\\0&0&b\\0&1&-1\end{bmatrix}, \ {\rm where}\  a,b\in\Z\  {\rm satisfy}\  |a|,|b|\leq 3.
$$  
\end{Prop}
\begin{proof} By Lemma~\ref{L:widthbound} we have $\w_{e_3}(Q)\leq 2$, that is, lattice points in $Q$ can have difference at most 2  in their $z$-coordinates. Assume first that no two lattice points in $Q$ have the same $z$-coordinate, so after a  translation we can assume that the $z$-coordinates are $-1,0,1$. Then we can unimodularly map $Q$ to $ \begin{bmatrix}0&0&a\\0&0&b\\0&1&-1\end{bmatrix}$ while fixing the $(x,y)$-plane. The intersection of  $Q$ with the $(x,y)$-plane is the segment $I$ that connects the origin to $(a/2,b/2,0)$. Hence $T_0+Q$ contains $[0,e_2]+I$ and the area of this parallelogram is $|a|/2$. If $|a|\geq 12$ the area is at least 6, so the parallelogram has at least
 9 lattice points, which contradicts $L(T_0+Q)=2$. We conclude that $|a|<12$ and, similarly, we get the same bound on $|b|$. Using {\sc Magma}  we run through all such values of $a$ and $b$ and check whether $L(T_0+Q)=2$. We conclude that
 $|a|, |b|\leq 3$.
 
 Next we consider the case when at least two lattice points of $Q$ have the same $z$-coordinate. Then by Lemma 1.8 from~\cite{SoSo1} it follows that the vector connecting these two points is $e_1$, $e_2$, or $e_1+e_2$. 
 Using a symmetry of $T_0$ we can assume that it is $e_1$. Then, 
 after a lattice translation and a unimodular transformation that fixes the $(x,y)$-plane, $Q$ has the form
 $$\begin{bmatrix}0&1&a\\0&0&b\\0&0&c\end{bmatrix}\ {\rm where}\  a,b, c\in\Z\  {\rm satisfy}\ 0\leq a,b<c\leq 2.
 $$
Checking using {\sc Magma}  whether we get $L(T_0+Q)=2$ for such values of $a,b$, and $c$ we conclude that $a=b=0$ and $c=1$, and the conclusion follows. 
Note that we have shown that if $Q$ has two lattice points in the plane $z=0$ then the third lattice point in $Q$ must be in the plane $z=1$ or $z=-1$. \end{proof}

\begin{Th}\label{T:T_0+four}
Let $Q\subset\R^3$ be a lattice polytope with $|Q|\geq 4$  and $L(T_0+Q)=2$. Then $|Q|=4$ and, after a lattice translation and  a unimodular transformation that maps $T_0$ to itself, $Q$ is equal to
$$\begin{bmatrix}0&0&3\\0&0&0\\0&1&-1\end{bmatrix} \ {\rm or} \begin{bmatrix}0&1&0&1\\0&0&0&2\\0&0&1&-1\end{bmatrix}.$$
Note that in the first case $Q$ is equivalent to $T_0$ and in the second case $Q$ is equivalent to~$S_2$.
\end{Th}

\begin{proof} 
 By Lemma~\ref{L:widthbound} we have $\w_{e_3}(Q)\leq 2$, and since  $|Q|\geq 4$, we know that $Q$ must have two lattice points with the same $z$-coordinate, which after a translation can be assumed to be 0. In 
Proposition~\ref{L:T_0plusQ} we showed that in this case $Q$ may have at most two 
additional lattice points, one in the plane $z=1$, and another in $z=-1$, which implies that  $|Q|=4$. Further, by our work in Proposition~\ref{L:T_0plusQ}, there exists a unimodular transformation fixing the  $(x,y)$-plane which brings $Q$ to
the form $\begin{bmatrix}0&1&0&a\\0&0&0&b\\0&0&1&-1\end{bmatrix}$, where $a,b\in\Z$ satisfy $|a|,|b|\leq 3.$ 
Using {\sc Magma}, we find that up to a lattice translation and  a unimodular transformation that maps $T_0$ to itself there are only  two options for $Q$, 
$$\begin{bmatrix}0&0&3\\0&0&0\\0&1&-1\end{bmatrix} \ {\rm and} \ \begin{bmatrix}0&1&0&1\\0&0&0&2\\0&0&1&-1\end{bmatrix}.$$
\end{proof}

For polytopes of width one, \rt{T_0+four} provides the following.

\begin{Cor}\label{C:width-one}
 Let $P\subset\R^3$ be of lattice width one. Then any maximal decomposition in $P$ contains at most one
 summand with 4 or more lattice points.
\end{Cor}
\begin{pf} After an affine unimodular transformation, we may assume that $P$ lies between the planes $z=0$ and $z=1$.
Suppose $P_1,P_2$ are two summands in a maximal decomposition in $P$ with $|P_1|\geq|P_2|\geq 4$.
Note that one of them, say $P_1$, must be 2-dimensional (otherwise $\w_{e_3}(P_1+P_2)>1$) and, hence, 
must be equivalent to $T_0$. Applying an affine unimodular transformation we may further assume that $P_1=T_0$.
But then, by \rt{T_0+four}, $P_2$ must be equivalent to a polytope whose width in the $z$-direction equals two, a contradiction.
\end{pf}

\subsection{$P$ has a subpolytope equivalent to $T_0$}\label{S:T0-facet}

We next consider the case when $P$ contains a proper subpolytope equivalent to $T_0$, hence, $|P|\geq 5$. As before, we first treat the case when $|P|=5$ and then show that for the case $|P|\geq 6$
there is no $Q$ with $|Q|\geq 3$ and $L(P+Q)=2$. In particular, this will imply that the subpolytope equivalent to $T_0$ must be a facet of $P$.

According to the classification of size 5 polytopes in~\cite{BlancoSantos1} the only such polytopes $P$ with
$|P|=5$ are the ones equivalent to
$$T_1=\left[\begin{matrix} 1&0&-1&0\\ 0&1&-1&0\\ 0&0&0&1\end{matrix}\right]\ \ {\rm or}\  \  T_2=\left[\begin{matrix} 1&0&-1&2\\ 0&1&-1&1\\ 0&0&0&3\end{matrix}\right].
$$  

\begin{Prop}\label{P:T_1}
Let $P=T_1$  and  let $Q$ be a lattice polytope such that $|Q|\geq 3$ and $L(Q)=1$. Suppose that  $L(P+Q)=2$. Then $|Q|=3$ and, up to a lattice translation, 
$$Q=\begin{bmatrix}0&1&0\\0&0&0\\0&0&1\end{bmatrix}, \begin{bmatrix}0&0&0\\0&1&0\\0&0&1\end{bmatrix},\  {\rm or}\ \begin{bmatrix}0&-1&0\\0&-1&0\\0&0&1\end{bmatrix}.
$$
\end{Prop}
\begin{proof} This is obtained by running  the functions {\tt AddTriangleHuh} and {\tt AddTetraHuh} for~$T_1$.
\end{proof}

\begin{Rem}\label{R:invariant} 
It is convenient to give an invariant description of the triangles $Q$ in the statement of \rp{T_1}. Let $P$ be a polytope equivalent to $T_1$. Denote the vertices of $P$ by
$v_1,\dots,v_4$ such that $[v_1,v_2,v_3]$ is the facet equivalent to $T_0$, and let $v_0$ be the center of this facet. Then  \rp{T_1} claims that if $Q$ is a polytope with $|Q|\geq 3$, $L(Q)=1$, and $L(P+Q)=2$ then 
$Q$ equals $[v_0,v_1,v_4]$, $[v_0,v_2,v_4]$, or $[v_0,v_3,v_4]$. 
\end{Rem}

\begin{Prop}\label{P:T_2}
Let $P=T_2$  and  let $Q$ be a lattice polytope such that $|Q|\geq 3$ and $L(Q)=1$. Then  $L(P+Q)\geq 3$. 
\end{Prop}
\begin{proof} 
The function {\tt AddTriangleHuh} for $T_2$ gives an empty output and the conclusion follows. 
\end{proof}

\begin{Prop}\label{P:T_0_or_T_1}
Let $P$ be a lattice polytope that contains a proper subpolytope equivalent to $T_0$
and  let $Q$ be a lattice polytope such that $|Q|\geq 3$. Suppose that $L(P)=L(Q)=1$ and $L(P+Q)=2$.
Then  $P$ is equivalent to $T_1$ and, in particular, $|P|=5$.
\end{Prop}
\begin{proof} 
If $|P|=5$ then by the classification in~\cite{BlancoSantos1} $P$ is equivalent to $T_1$ or $T_2$, where the latter case is ruled out by \rp{T_2}.
Suppose  $|P|\geq 6$.
Let $[v_1,v_2,v_3]$ be the subpolytope equivalent to $T_0$ and let $v_0$ be its center. Let $v_4$ and $v_5$ be two other lattice points in $P$.
Then both $[v_1,v_2,v_3,v_4]$ and $[v_1,v_2,v_3,v_5]$ must be equivalent to $T_1$. By
\rr{invariant}, this implies that $Q\in \{[v_0,v_1,v_4], [v_0,v_2,v_4], [v_0,v_3,v_4]\}$
and, at the same time, $Q\in\{[v_0,v_1,v_5], [v_0,v_2,v_5], [v_0,v_3,v_5]\}$, which is impossible.
\end{proof}

\subsection{$P$ is clean, but not empty}

We next work on the remaining case when $P$ is clean, but contains at least one interior lattice point. Recall that {\it clean} means that the only boundary lattice points of $P$ are its vertices. Since we have already considered the case when $T_0$ is a subpolytope of $P$ we can further
assume that no four points in $P$ are coplanar. 

First, assume that $P$ is a clean lattice tetrahedron with one interior lattice point, i.e. a {\it Fano tetrahedron}. In the next proposition we use Kasprzyk's classification of Fano tetrahedra, see~\cite[Section 2]{Kasp}.

\begin{Prop}\label{P:Fano}
Let $P\subset\R^3$ be a Fano tetrahedron and let $Q\subset\R^3$ be a lattice polytope such that $L(Q)=1$ and $L(P+Q)=2$.  If $|Q|\geq 3$ then  $P$ is equivalent to either $K_1$ or $K_2$, as defined in Table~\ref{table}.
\end{Prop}
\begin{proof} 
By the classification in~\cite{Kasp}, up to lattice equivalence, there are eight Fano tetrahedra. For each such $P$, we run function  {\tt AddTriangleHuh} and it returns a negative answer for all of them except for 
$K_1$ and $K_2$. \end{proof}

\begin{Prop}\label{P:Fano-K1-K2}
Let $P=[v_1,\dots,v_4]\subset\R^3$ be a Fano tetrahedron and $v_0$ its interior lattice point.
Let $Q\subset\R^3$ be a lattice polytope with  $|Q|=4$, $L(Q)=1$ and $L(P+Q)=2$. 
\begin{enumerate}
\item If $P$ is equivalent to $K_1$ then $Q$ equals $[v_0,v_1,v_2,v_3]$, $[v_0,v_1,v_2,v_4]$,
$[v_0,v_1,v_3,v_4]$, or $[v_0,v_2,v_3,v_4]$. Moreover, $(P,Q)$ is equivalent to $(K_1,S_1)$.
\item If  $P=K_2$ then 
$Q=\begin{bmatrix}0&0&1&1\\0&1&0&1\\0&0&0&2\end{bmatrix}$, which is equivalent to $S_2$.
\end{enumerate}
\end{Prop}
\begin{proof} 
We  run the function {\tt AddTetraHuh} for $K_1$ and $K_2$ and obtain the claimed output. The tetrahedron $Q$ that we obtain in (2) is equivalent to $S_2$:
 $$\begin{bmatrix}-1&0&0\\0&-1&0\\-1&-1&1\end{bmatrix}\cdot\begin{bmatrix}0&0&1&1\\0&1&0&1\\0&0&0&2\end{bmatrix}+\begin{bmatrix}1\\1\\1\end{bmatrix}=\begin{bmatrix}1&1&0&0\\1&0&1&0\\1&0&0&1\end{bmatrix}=S_2.
 $$ 

For (1) we also note that there exists a unimodular map preserving $K_1$ which  maps the convex hull of the origin and any three vertices of $K_1$ to  $S_1$.
\end{proof}

We next  work on the case when $P$ has more than one lattice point in the interior. We can then use the following statement to reduce to the case of a Fano tetrahedron.

\begin{Lemma}\label{L:FanoLem}
Let $P$ be a lattice polytope in $\mathbb{R}^3$ such that no four lattice points in $P$ are coplanar.  If $P$ has an interior lattice point $p$, then $P$ contains a subpolytope which is a Fano tetrahedron with interior lattice point  $p$.
\end{Lemma}
\begin{proof}
By Carath\'eodory's Theorem \cite[I.2.3.]{Bar} we can find four lattice points $p_1, p_2, p_3,p_4$ in $P$ such that the convex hull $[p_1,p_2,p_3,p_4]$ contains $p$ as an interior point.  If $p$ is the only such interior lattice point, then we are done, if there is another interior lattice point $q$ then both $p$ and $q$ can be expressed as convex combinations of $p_1,p_2,p_3,$ and $p_4$  
\begin{align*}
p&=\alpha_1p_1+ \alpha_2p_2+\alpha_3p_3+\alpha_4p_4,\\
q&=\beta_1p_1+\beta_2p_2+\beta_3p_3+\beta_4p_4.
\end{align*}
Assume without loss of generality that $\frac{\alpha_4}{\beta_4}\leq \frac{\alpha_i}{\beta_i}$ for $i=1,\dots,4$. Multiplying the second equality by $\alpha_4/\beta_4$ and subtracting the result from the first, we get 
$$p=\frac{\alpha_1\beta_4-\alpha_4\beta_1}{\beta_4}p_1+\frac{\alpha_2\beta_4-\alpha_4\beta_2}{\beta_4}p_2+\frac{\alpha_3\beta_4-\alpha_4\beta_3}{\beta_4}p_3+\frac{\alpha_4}{\beta_4}q,$$
which is a convex combination of $p_1,p_2,p_3,$ and $q$.  Note that none of the coefficients is zero for otherwise $P$ would contain four coplanar lattice points. Thus $p$ is an interior point of the convex hull $[p_1,p_2,p_3,q]$.   We next pass to this tetrahedron and continue the process until $p$ is the only interior lattice point.
\end{proof}

We next generalize the result of Lemma~\ref{L:empty_tetrahedra} to the case when lattice tetrahedron $P$ is not necessarily empty.

\begin{Lemma}\label{L:volume_bound}
Let $P$ be a lattice tetrahedron and $Q$ be a lattice polytope such that  $|Q|\geq 3$,  $L(P)=L(Q)=1$ and $L(P+Q)=2$. Then $\Vol_3(P)\leq 5$.
\end{Lemma}
\begin{proof} 
We have $4\leq |P|\leq 8$ and the case when $|P|=4$ is covered in Lemma~\ref{L:empty_tetrahedra}. We now argue by induction on $|P|$. We assume that the conclusion holds true for a tetrahedron with $n$ lattice points and assume that $|P|=n+1$. We pick a lattice point in $P$ which is not a vertex and connect it to the four vertices of $P$ to get a partition of $P$ into 3 or 4 tetrahedra (depending on whether the point is on the surface or in the interior of $P$). Each of these new tetrahedra has at most $n$ lattice points, so by the induction assumption the volume of each of them is at most 5, and hence the volume of $P$  is bounded from above by 20. 

The facets of $P$ are either unit triangles or triangles equivalent to $T_0$. Let the vertices of $P$ be $v_1,\dots,v_4$. If the facets $[v_1,v_2,v_3]$ and $[v_1,v_2,v_4]$ are equivalent to $T_0$ then the lattice points inside these facets are $(v_1+v_2+v_3)/3$ and $(v_1+v_2+v_4)/3$, but this implies that their difference  $(v_3-v_4)/3$ has integer coordinates and, hence, the edge $[v_3,v_4]$ is not primitive, which contradicts $L(P)=1$. We have shown that at most one facet of $P$ is equivalent to $T_0$, so the remaining ones are unit triangles. We can now apply a unimodular map to $P$ so that one of such facets becomes $[0,e_1,e_2]$ and, hence,
$P=[0,e_1,e_2, (a,b,c)]$ 
for some $(a,b,c)\in\Z^3$ such that $0\leq a\leq b\leq c\leq 20$.  For each such $P$ with $c\geq 6$ we run function {\tt AddTriangleHuh}. The output is empty, and the conclusion follows.
\end{proof}

\begin{Prop}\label{P:rule_out_K1}
Let $P,Q\subset\R^3$ be lattice polytopes such that $L(P)=L(Q)=1$ and $|Q|\geq 3$. Suppose that $P$ properly contains $K_1$. Then $L(P+Q)\geq 3$.
\end{Prop}

\begin{proof} 
Let $(a,b,c)$ be a vertex of $P$ which does not belong to $K_1$. Then the tetrahedron $[0,e_1,e_2, (a,b,c)]$
is a subpolytope of $P$, so
using Lemma~\ref{L:volume_bound} we conclude $|c|\leq 5$. Similarly, we have $|a|\leq 5$ and $|b|\leq 5$.
For each such $(a,b,c)$ we work with the subpolytope 
$$R=\begin{bmatrix}1&0&0&-1&a\\0&1&0&-1&b\\0&0&1&-1&c\end{bmatrix}$$ 
of $P$.
We check whether $L(R)=1$ and if this is the case we run function {\tt AddTriangleHuh}, applied to $R$.  We observe that the only returned values are the five lattice points $(a,b,c)\in K_1$ and the conclusion follows.
\end{proof}

\begin{Prop}\label{P:rule_out_K2}
Let $P,Q\subset\R^3$ be lattice polytopes such that $L(P)=L(Q)=1$ and $|Q|\geq 3$. Suppose that $P$ properly contains $K_2$. Then $L(P+Q)\geq 3$.
\end{Prop}

\begin{proof} 
Let $(a,b,c)$ be a vertex of $P$ which is not in $K_2$. Then $R=\begin{bmatrix}1&0&1&-1&a\\0&1&1&-1&b\\0&0&2&-1&c\end{bmatrix}$ is a subpolytope of $P$.
Using Lemma~\ref{L:volume_bound} and considering the simplices in $R$ that $(a,b,c)$ generates together with other vertices of $R$ (including the origin) we get
$$|c|\leq 5,\ \ |2b-c|\leq 5,\ \ |2a-c|\leq 5,\ \ |b+c|\leq 5,\ \ |a+c|\leq 5,
$$
which implies that $|a|\leq 3, |b|\leq 3$, and $|c|\leq 5$. For each $R$ with such restriction on $a,b$, and $c$ we check whether $L(R)=1$.  If this is the case we run function   {\tt AddTriangleHuh} applied to $R$ and arrive at the stated conclusion. 
\end{proof}

\subsection{Summary of the results}\label{S:summary} 
In the next three theorems  we summarize the classification according to the
sizes of the polytopes in the pair: $(|P|,|Q|)=(5,4), (5,5), (6,3)$. In fact, we show that the latter case we have $L(P+Q)\geq 3$. In particular, this implies that if a maximal decomposition contains a polytope of size larger than 5 then the other summands must be lattice segments. Note that the cases $(|P|,|Q|)=(4,3), (4,4)$ are covered by \rr{empty+3}, 
\rt{two_empty_tetrahedra}, Proposition~\ref{L:T_0plusQ}, and \rt{T_0+four}.

\begin{Th}\label{T:5plus4}
Let $P,Q\subset\R^3$ be lattice polytopes with $L(P)=L(Q)=1$, $L(P+Q)=2$,  $|P|=5$, and  $|Q|=4$. Then the pair $(P,Q)$ is unimodularly equivalent to $(K_1, S_1), (K_2,S)$, or $(E,S_2)$, where  
$S=\begin{bmatrix} 0 & 0 & 1 & 1\\ 0 & 1 & 0 & 1\\ 0 & 0 & 0 & 2 \end{bmatrix}$ is equivalent to $S_2$.

Further, if (up to a lattice translation) $P=E$ then $Q=S_2$. Similarly, if $P=K_2$ then $Q=S$, but if $P=K_1$ there are four options for $Q$. Namely,  $Q$ is one of the four empty tetrahedra properly contained in $K_1$.
\end{Th}

\begin{proof}
If $P$ is an empty polytope, by Theorem~\ref{T:empty_5points} we get the pair $(E,S_2)$. If $P$ contains a facet equivalent to $T_0$ then  $P$ is equivalent to either $T_1$ or $T_2$ and by Propositions~\ref{P:T_1} and \ref{P:T_2} there is no such $Q$. Finally, if $P$ is a Fano tetrahedron then Proposition~\ref{P:Fano-K1-K2} gives the first two pairs on the list. 
\end{proof}

\begin{Th}\label{T:5plus5} Let $P,Q\subset\R^3$ be lattice polytopes with $L(P)=L(Q)=1$, $L(P+Q)=2$,  $|P|=5$, and  $|Q|=5$. Then the pair $(P,Q)$ is  equivalent to $(K_1, K_1)$.
\end{Th}
\begin{proof} By Theorem~\ref{T:5plus4} each $P$ and $Q$ is individually equivalent to $K_1$, $K_2$, or $E$. Further, if $P$ is equivalent to $E$ or $K_2$ then there is only one option for a $4$-point polytope that can appear in a maximal decomposition with $P$ and, since, a 5-point $Q$ would contain more than one 4-point subpolytope, we conclude that each $P$ and $Q$ is equivalent to $K_1$. Let $P=K_1$ then by  Theorem~\ref{T:5plus4} the only  lattice 4-polytopes that can appear with $P$ in a maximal decomposition are the four empty tetrahedra properly contained in $K_1$. Then the standard way of assembling these tetrahedra to form $K_1$ is the only way a 5-point lattice polytope can be formed using these, and we conclude that up to a lattice translation $P=Q=K_1$. We use {\sc Magma} to check that $L(K_1+K_1)=2$.
\end{proof}

Finally, we show that lattice polytopes $P$ and $Q$ with $|P|=6$ and $|Q|=3$ cannot appear together in a maximal decomposition.

\begin{Th}\label{T:6plus3} Let $P,Q\subset\R^3$ be lattice polytopes such that $|P|=6$ and $|Q|=3$. Then $L(P+Q)\geq 3$.
\end{Th}
\begin{proof}
The case when $P$ is empty is covered by Proposition~\ref{P:empty_6points}. If $P$ contains a subpolytope equivalent to $T_0$ the conclusion follows from  Proposition~\ref{P:T_0_or_T_1}.
If $P$ is clean then it contains two interior lattice points. Then by Lemma~\ref{L:FanoLem}  it properly contains a Fano tetrahedron which by  Proposition~\ref{P:Fano} is equivalent to $K_1$ or $K_2$ and these cases are then covered  by  Propositions~\ref{P:rule_out_K1} and \ref{P:rule_out_K2}.
\end{proof}


\section{Classifying triples $(P,Q,R)$ with $L(P+Q+R)=3$}

Our next goal is to explore how three (and more) lattice polytopes with at least four lattice points each can appear together in a maximal decomposition.

\begin{Lemma}\label{L:three_segments}  Let $I_1$, $I_2$,  and $I_3$ be three lattice segments in $\R^3$ with the corresponding linearly independent direction vectors $u_1,u_2$, and $u_3$ such that $L(I_1+I_2+I_3)=3$. Let $v$ be the primitive normal vector to the  plane spanned by $u_1$ and $u_2$.
Then there are two possible cases:
\begin{itemize}
\item[(1)] $I_1+I_2$ is equivalent to the  unit square $[0,e_1,e_2,e_1+e_2]$,
and then $\w_v(I_3)\leq 9$;
\item[(2)] $I_1+I_2$ is equivalent to the parallelogram $[0,e_1,e_1+2e_2,2e_1+2e_2]$,
and then $\w_v(I_3)\leq~4$.
\end{itemize}
\end{Lemma}
\begin{proof}
We can assume that $u_1=e_1$ and $u_2=(a,b,0)$ where $0\leq a< b$, and hence $v=e_3$. By 
\cite[Lemma 1.7]{SoSo1} we have $b\leq 2$ which implies that either $b=1$ and $a=0$, or $b=2$ and $a=1$.
Note that in the first case  $I_1+I_2$ is equivalent to the  unit square  $[0,e_1,e_2,e_1+e_2]$,
and in the second  to the parallelogram $[0,e_1,e_1+2e_2,2e_1+2e_2]$.
We can now assume that $u_3=(a,b,c)$, where $0\leq a\leq b< c$.
Applying Lemma~\ref{L:widthbound} to $P=I_1+I_2$ and $Q=I_3$ we get $c=\w_{e_3}(I_3)\leq 14$. Next we use {\sc Magma} to check whether $L(I_1+I_2+I_3)=3$ for all such $a,b$, and $c$ and improve the bound 
to $c=\w_{e_3}(I_3)\leq 9$ in case (1) and to $c=\w_{e_3}(I_3)\leq 4$ in case (2).
\end{proof}

\begin{Prop}\label{P:T_0QR}
Let $Q,R\subset\R^3$ be lattice polytopes with $|Q|\geq 3$ and $|R|\geq 3$ such that $L(T_0+Q+R)=3$. Then $|Q|=|R|=3$ and  up to a lattice translation 
$$Q=R=\begin{bmatrix}0&1&0\\0&0&0\\0&0&1\end{bmatrix},\ \ \begin{bmatrix}0&0&0\\0&1&0\\0&0&1\end{bmatrix},\ {\rm or}\ \begin{bmatrix}0&1&0\\0&1&0\\0&0&1\end{bmatrix}.
$$

\end{Prop}
\begin{proof} Suppose first that $|Q|=|R|=3$. 
In light of Proposition~\ref{L:T_0plusQ} we first assume that $Q=\begin{bmatrix}0&0&a\\0&0&b\\0&1&-1\end{bmatrix}$, where $a,b\in\Z$ satisfy $|a|,|b|\leq 3.$ Then $T_0+Q$ contains the sum of primitive segments with direction vectors 
$e_1$ and $e_3$, so by Lemma~\ref{L:three_segments} we conclude that if $(p,q,r)$ is a direction vector of a primitive segment in $R$ then $|q|\leq 9$. Similarly, considering $e_2$ in the place of $e_1$ we get $|p|\leq 9$.
Next, using part~(2) of Lemma~\ref{L:widthbound} we conclude that $|r|\leq 2$.

For $|a|,|b|\leq 3$ we next use {\sc Magma} to find all directions $(p,q,r)$ with $|p|\leq 9, |q|\leq 9, |r|\leq 2$ such that for $I=[(0,0,0), (p,q,r)]$ we have $L(T_0+Q+I)=3$ and then check whether such segments $I$ can be used to form a lattice triangle $R$ with $L(T_0+Q+R)=3$.
The computation confirms  that there are no such lattice triangles and hence by Proposition~\ref{L:T_0plusQ} we conclude that there are three options for each $Q$ and $R$, namely, 
$$\begin{bmatrix}0&1&0\\0&0&0\\0&0&1\end{bmatrix},\ \ \begin{bmatrix}0&0&0\\0&1&0\\0&0&1\end{bmatrix}\!,\  {\rm and} \ \begin{bmatrix}0&1&0\\0&1&0\\0&0&1\end{bmatrix}.
$$
Checking using {\sc Magma} whether $L(T_0+Q+R)=3$ for each of these options we conclude that up to a lattice translation $Q=R$. 
The general case with $|Q|\geq 3$ and $|R|\geq 3$ follows.
\end{proof}

\begin{Prop}\label{P:threeS_1} 
Suppose that lattice polytopes $P,Q,R\subset\R^3$ satisfy $L(P+Q+R)=3$ and $|P|=|Q|=|R|=4$. If two out of these three
polytopes are individually equivalent to $S_1$ then, up to a lattice translation, we have $P=Q=R=S_1$.
\end{Prop}
\begin{proof}
Let $P$ and $Q$ be each equivalent to $S_1$ and assume $P=S_1$.
We run the function {\tt FindTetra} for $S_1$ and record the output. We then  cycle through $Q$ and $R$ in this output list where we only consider $Q$ of volume one
and check if $L(S_1+Q+R)=3$.  Up to a lattice translation, the only $Q$ and $R$ that satisfy this property are equal to $S_1$.
\end{proof}

\begin{Th}\label{T:main} Suppose that lattice polytopes $P,Q,R\subset\R^3$  with $|P|, |Q|, |R|\geq 4$ satisfy $L(P+Q+R)=3$.
Then, up to reordering and individual lattice translations, there are five options for $(P,Q,R)$, where in the first three options we have $|P|=|Q|=|R|=4$ and in the fourth one $|P|=5$ and $|Q|=|R|=4$:
\begin{itemize}
\item[(i)]  $P=Q=R$ and equivalent to $S_1$;
\item[(ii)] $P$ is equivalent to $S_1$, and $Q=R$ and equivalent to $S_2$;
\item[(iii)] $P=Q=R$ and equivalent to $S_2$;
\item[(iv)]  $(P,Q,R)$ is equivalent to $(E,S_2,S_2)$;
\end{itemize}
\end{Th}

\begin{proof}
Suppose first that $|P|=|Q|=|R|=4$. The case when $P=T_0$ is ruled out by Proposition~\ref{P:T_0QR}. Hence $P,Q$, and $R$ are empty tetrahedra and by Theorem~\ref{T:two_empty_tetrahedra} each of them is individually equivalent to $S_1$ or $S_2$.
If two of them are equivalent to $S_1$ then by Proposition~\ref{P:threeS_1} we have $P=Q=R=S_1$. Also,
 $L(kS_1)=k$ for $k\in\mathbb{N}$, see \rs{mink-length}.

If two are equivalent to $S_2$ then by Proposition~\ref{P:twoS_2} they are the same up to translation. Also, if all three are equivalent to $S_2$ then  $P=Q=R=S_2$. We confirmed using {\sc Magma} that $L(3S_2)=3$.

By Theorem~\ref{T:6plus3} we have $|P|,|Q|, |R|\leq 5$. If, say, $|P|=5$ then  such $P$ and options for 4-point  $Q$ and $R$ are described in Theorem~\ref{T:5plus4}. Let first $P=K_1$. Then each  $Q$ and $R$ is equivalent to $S_1$, but this then contradicts Proposition~\ref{P:threeS_1} since $K_1$ contains two different copies of $S_1$. Further, if  $P=K_2$ then $Q=R=S$, where $S$ is equivalent to $S_2$. We rule this case out by checking using {\sc Magma} that
$L(K_2+2S)>3$. If $P=E$ we get  $Q=R=S_2$ and using {\sc Magma} we confirm $L(E+2S_2)=3$.

If $|P|=|Q|=5$ and $|R|=4$, by Theorems~\ref{T:5plus4} and \ref{T:5plus5} we conclude that $|P|=|Q|=K_1$ and $R$ is one of the four copies of $S_1$  properly contained in $K_1$. 
We confirm using {\sc Magma} that $L(K_1+K_1+S_1)>3$.

Finally, if $|P|=|Q|=|R|=5$ then $P=Q=R=K_1$ and we get $L(P+Q+R)>3$ since $K_1$ contains the segment $I=[(-1,-1,-1), (1/3,1/3,1/3)]$  and $L(3I)=4$.
\end{proof}

The following statement follows immediately from \rt{main}.

\begin{Cor}\label{C:4-and-more}
 Let $P_1,\dots,P_k\subset\R^3$  be $k\geq 3$ lattice polytopes  with at least 4 lattice points each such that $L(P_1+\cdots+P_k)=k$. Then up to reordering and individual lattice translations there are at most four options for $(P_1,\dots,P_k)$:
\begin{itemize}
\item[(i)]  $P_1=\cdots=P_k$ and equivalent to $S_1$;
\item[(ii)] $P_1$ is equivalent to $S_1$, and $P_2=\cdots=P_k$ and equivalent to $S_2$;
\item[(iii)] $P_1=\cdots=P_k$ and equivalent to $S_2$;
\item[(iv)]  $(P_1,\dots,P_k)$ is equivalent to $(E,S_2,\dots,S_2)$.
\end{itemize}
\end{Cor}

\begin{Rem} 
Note that we only claim that whenever $L(P_1+\cdots+P_k)=k$ for $k\geq 3$, we are in one of the cases (i) through (iv), but we do not claim that in each of these cases we indeed have $L(P_1+\cdots+P_k)=k$. Although this is true in case (i), since $L(kS_1)=k$ for all $k\geq 1$, we
do not know whether this is true in cases (ii)--(iv). 
\end{Rem}


\section{Number of $\F_q$-zeros of polynomials with $L(P_f)=1$}

In this section we give an upper bound for the number of $\F_q$-zeros of Laurent polynomials $f\in\F_q[x^{\pm 1},y^{\pm 1},z^{\pm 1}]$ whose Newton polytopes $P_f$ have Minkowski length one. We start with a proposition which deals with polynomials that are linear in $z$.

\begin{Prop}\label{P:Infinite} 
Consider a polynomial $f=f_0+zf_1$ 
for some $f_0,f_1\in\F_q[x^{\pm},y^{\pm}]$ without common factors. 
Let $P$, $P_0$, and $P_1$ be the Newton polytopes of $f$, $f_0$, and $f_1$, respectively. 
Then 
$$N_f\leq (q-1)^2+\left(\Vol_3(P)-\Vol_2(P_0)-\Vol_2(P_1)\right)q-N_0-N_1,$$
where $\Vol_3$ and $\Vol_2$ denote the normalized 3- and 2-dimensional volumes, and $N_i$ is the number of $\F_q$-zeros of
$f_i$ in $(\F_q^*)^2$, for $i=0,1$.
\end{Prop}

\begin{pf} Let $Z_{i}\subset(\F_q^*)^2$ be the set of $\F_q$-zeros of $f_i$, $i=0,1$, and let
$Z_{0,1}=Z_0\cap Z_1$.
Note that for every $(x,y)\in Z_{0,1}$ and any $z\in\F_q^*$ the triple $(x,y,z)$ is a zero of $f$. Furthermore, for any $(x,y)\in(\F_q^*)^2\setminus Z_{0}\cup Z_{1}$ there exists a unique $z\in\F_q^*$ such that the triple $(x,y,z)$ is a zero of $f$. Any other triple $(x,y,z)\in (\F_q^*)^3$ cannot be a zero of $f$. Therefore,
$$N_f=|Z_{0,1}|(q-1)+(q-1)^2-|Z_0|-|Z_1|+|Z_{0,1}|.$$

Our next step is to bound $|Z_{0,1}|$ from above by the mixed volume $V(P_0,P_1)$. 
Let $C_i$ be the algebraic curve defined by $f_i=0$ in the toric compactification of $(\bar\F_q^*)^2$ corresponding to $P_0+P_1$. 
Note that since $f_0$ and $f_1$ do not have common factors, 
the intersection $ C_0\cap  C_1$ is zero-dimensional. 
Therefore,  by the BKK bound \cite[Th B]{Be} we have
 $$|Z_{0,1}|\leq | C_0\cap  C_1|\leq V(P_0,P_1).$$
 It remains to use the formula in \rl{mix} which relates the mixed volume $V(P_0,P_1)$ and the normalized volume of $P$.
\end{pf}

The following result appears in Whitney's PhD thesis \cite{Josh}. Its proof is based on the Grothendieck-Lefschetz trace formula and cohomology computation for hypersurfaces $H_f$ in a toric variety $X_P$, where $P$ is a lattice polytope from one of the 108 classes
of polytopes with $L(P)=1$ and lattice width greater than one, and $f$ is a Laurent polynomial with Newton polytope $P$.

\begin{Th}\cite[Th 4.30]{Josh}\label{T:JoshFinite} 
Assume $\cchar(\F_q)>41$. Let $f$ be a Laurent polynomial whose Newton polytope 
$P$ satisfies $L(P)=1$ and ${\rm w}(P)>1$. Then 
$$N_f\leq (q-1)^2+(\Vol_3(P)-F(P)/2)q+F(P)/2,$$
where $\Vol_3(P)$ is the normalized 3-dimensional volume and $F(P)$ is the number of facets of~$P$.
 \end{Th}

\begin{Rem}\label{R:char} 
Whitney's classification is split according to the number of vertices and the number of interior lattice points of $P$, rather than the lattice width of $P$. In fact, the ``finite cases'' in \cite[Th 4.30]{Josh} consist of 109 classes of lattice polytopes, 1 of which has  lattice width one (the class containing $T_1$, see \rs{T0-facet}) and the other 108 have lattice width greater than one.
\end{Rem}

\begin{Rem}\label{R:char41} The condition on the field characteristic ensures that the corresponding hypersurface $H_f$ has at worst isolated singularities, \cite[Cor 4.18]{Josh}. For specific polytopes $P$ this condition can be relaxed. For example, when $P$ equals $K_1$ or  $K_2$ it is enough to require $\cchar(\F_q)\neq 2,3$.
\end{Rem}

\begin{Th}\label{T:VolumeBound} 
Assume $\cchar(\F_q)>41$. Let $f$ be a Laurent polynomial whose Newton polytope 
$P$ is 3-dimensional and has Minkowski length one. Then 
$$N_f\leq (q-1)^2+(\Vol_3(P)-2)q+2,$$
where $\Vol_3(P)$ is the normalized 3-dimensional volume of $P$.
\end{Th}

\begin{pf} This follows from \rp{Infinite} and \rt{JoshFinite}. Indeed, if $P$ has lattice width greater than one then
the bound  follows directly from \rt{JoshFinite} as $P$ has at least 4 facets. Thus we may assume that $P$ has lattice width one.
After an affine unimodular transformation which corresponds to a monomial change of variables in $f$ we may assume that
$P$ and $f$ have the form as in \rp{Infinite}. If $f_0$ and $f_1$ had a common factor then 
$P_0$ and $P_1$ would contain a common Minkowski summand. But then 
$P_0$ and $P_1$ would have a pair of parallel sides which forces
$P$ to contain a parallelogram. This contradicts the assumption $L(P)=1$. Therefore,
$f_0$ and $f_1$ do not have common factors and the conditions of \rp{Infinite} are satisfied.
It remains to show that the bound in \rp{Infinite} is no greater than $(q-1)^2+(\Vol_3(P)-2)q+2$,
i.e.,
\begin{equation}\label{e:prop-6.1}
\left(2-\Vol_2(P_0)-\Vol_2(P_1)\right)q\leq N_0+N_1+2.
\end{equation}

First, each $P_i$ is either a point or one of the three polytopes of Minkowski length one: a primitive segment, a unimodular
triangle, or equivalent to $T_0$. If at least one of the $P_i$ is equivalent to $T_0$ then \re{prop-6.1}
is trivial, as $\Vol_2(T_0)=3$.
We have the following table of values of $\Vol_2(P_i)$ and $N_i$ in the other three cases.

\renewcommand{\arraystretch}{1.2}
\begin{center}
\begin{tabular}{|l|c|c|}
\hline
$P_i$ & $\Vol_2(P_i)$ & $N_i$\\
\hline
\hline
{\rm point}& 0 & 0\\
\hline
{\rm  primitive segment}& 0 & $q-1$\\
\hline
{\rm unit triangle }& 1 & $q-2$\\
\hline
\end{tabular}
\end{center}
It is straightforward to check that \re{prop-6.1} holds for all possible pairs $(P_0,P_1)$
with parameters as in the above table. Note that since $P$ is 3-dimensional the pairs (point, point), (point, segment), and (segment, point) cannot happen.
\end{pf}

\begin{Prop}\label{P:special} 
Let $P$ be the Newton polytope of a Laurent polynomial $f\in\F_q[x^{\pm 1},y^{\pm 1},z^{\pm 1}]$. 
Assume $\cchar(\F_q)\neq 2,3$. We have the following upper bounds on $N_f$
\renewcommand{\arraystretch}{1.2}
\begin{center}
\begin{tabular}{|l|c|}
\hline
$P$ & {\rm upper bound on} $N_f$\\
\hline
\hline
{\rm  primitive segment}& $(q-1)^2$\\
\hline
{\rm unit triangle  }&  $(q-1)(q-2)$\\
\hline
{\rm unit 3-simplex  }&  $(q-1)^2-q+2$\\
\hline
{\rm equivalent to $T_0$  }&  $(q-1)(q+\floor{2\sqrt{q}}-2)$\\
\hline
{\rm equivalent to $S_2$   }&  $(q-1)^2+2$\\
\hline
{\rm equivalent to $E$  }&  $(q-1)^2+3$\\
\hline
{\rm equivalent to $K_1$   }&  $(q-1)^2+2q+2$\\
\hline
{\rm equivalent to $K_2$  }&  $(q-1)^2+3q+2$\\
\hline
\end{tabular}
\end{center}
\end{Prop}

\begin{pf}
The first three bounds are standard and, in fact, are exact values of $N_f$. This is because 
after a monomial change of variables (which preserves $N_f$) the polynomial $f$ can be brought to
the form $c_0+c_1x$, $c_0+c_1x+c_2y$, and $c_0+c_1x+c_2y+c_3z$, respectively, for some 
$c_i\in\F_q^*$.  The fourth bound follows directly from the Hasse-Weil bound as in  \cite[Proposition 2.1]{SoSo1}. The fifth and the sixth bounds follow from \rp{Infinite} as both $S_2$ and $E$ have width one. Finally, the seventh and the eighth bounds follow from \rt{JoshFinite} and \rr{char41} using 
$\Vol_3(K_1)=4$ and $\Vol_3(K_2)=5$.
\end{pf}

\section{Number of $\F_q$-zeros of polynomials in $\cL_P$}
In this section we come to the main results of the paper which provide a solution to Problems~\ref{Pr:all} and~\ref{Pr:max} for $n=3$.
As before, we fix a lattice polytope $P\subset\R^3$ and consider the space $\cL_P$ of Laurent polynomials  whose Newton polytopes are 
contained in $P$.  Let $N_P$ denote  the largest number of $\F_q$-zeros in $(\F_q^*)^3$ over all non-zero $f\in\cL_P$, i.e.
$$N_P=\max\{ N_f : 0\neq f\in\cL_P\}.$$
Our goal is to come up with a bound for $N_P$ which depends on $q$ and the Minkowski length of $P$.
As in \rpr{max}, we first look at polynomials in $\cL_P$ that have the largest number of absolutely irreducible factors.

\begin{Th}\label{T:max-a} Assume $\cchar(\F_q)>41$. 
Let $P\subset\R^3$ be a lattice polytope of Minkowski length $L$.
Consider $f\in\cL_P$ with the largest number of absolutely irreducible factors. Let $k$ be the number of those factors  with $4$ or more monomials. Then
\begin{enumerate}
\item if $k=0$ then $N_f\leq L\,(q-1)^2$;
\item if $k=1$ then 
\begin{enumerate}
\item $N_f\leq L\,(q-1)^2+(q-1)(\floor{2\sqrt{q}}-1)$, if $f$ has a factor with Newton polytope equivalent to $T_0$, 
\item $N_f\leq L\,(q-1)^2+(\Vol_3(P)-3L+1)q+2$, otherwise;
\end{enumerate}
\item if $k=2$ then $N_f\leq L\,(q-1)^2+2(q-1)(\floor{2\sqrt{q}}-1)$;
\item if $k\geq 3$ then $N_f\leq L\,(q-1)^2+2k+1\leq L\,(q-1)^2+2L+1$.
\end{enumerate}
\end{Th}

\begin{pf} Let $f=f_1\cdots f_L$ be the factorization of $f$ into $L$ absolutely irreducible factors.
Recall that it corresponds to a maximal decomposition $P_1+\dots+P_L\subseteq P$,
where $P_i$ is the Newton polytope of $f_i$. Ignoring possible common zeros of the $f_i$ we
have $N_f\leq N_{f_1}+\dots+N_{f_L}$ so it is enough to bound each $N_{f_i}$ separately. 

In case (1) each $P_i$ is either a primitive segment or a unit triangle. Using the first two rows in the table of \rp{special} we obtain the bound. In case (2a) the only summand of size four is equivalent to $T_0$ and the others are either primitive segments or unit triangles, and the bound follows from \rp{special}.

For (2b), assume that $|P_1|\geq 4$ and $|P_i|\leq 3$ for $i\geq 2$. By \rt{VolumeBound}, 
$$N_{f_1}\leq (q-1)^2+(\Vol_3(P_1)-2)q+2.$$ 
Also, by \rl{VolumeBound}, we have
$$\Vol_3(P)\geq \Vol_3(P_1+\dots+P_L)\geq \Vol_3(P_1)+3(L-1).$$
Finally, for $i\geq 2$ we have $N_{f_i}\leq (q-1)^2$ by \rp{special}. Combining these inequalities
we obtain the required bound.

For (3) we use our classification from \rs{classification-2}. Up to reordering the $P_i$, we may assume 
that  $|P_1|\geq |P_2|\geq 4$ and $|P_i|\leq 3$ for $i\geq 3$. Then $N_{f_i}\leq (q-1)^2$ for $i\geq 3$ and
we need to show 
\begin{equation}\label{e:sum-of-two}
N_{f_1}+N_{f_2}\leq 2(q-1)^2+2(q-1)(\floor{2\sqrt{q}}-1).
\end{equation}
As we saw in \rs{summary}, the 
only possible pairs of sizes $(|P_1|,|P_2|)$ are $(4,4)$, $(5,4)$, and $(5,5)$. In the first case,
by \rt{two_empty_tetrahedra} and \rt{T_0+four}, each of $P_1$ and $P_2$ is equivalent to $S_1$, $S_2$, or $T_0$.
By comparing the bounds in \rp{special}, the largest bound for $N_{f_1}+N_{f_2}$ is $2(q-1)(q+\floor{2\sqrt{q}}-2)$ when both $P_1$ and $P_2$ are equivalent to $T_0$, and \re{sum-of-two} follows. 
The case $(|P_1|,|P_2|)=(5,4)$ is covered by \rt{5plus4}. In this case the 
largest bound for $N_{f_1}+N_{f_2}$ is $2(q-1)^2+3q+4$ when $P_1$ is equivalent to $K_2$ and $P_2$ is equivalent to $S_2$. But this bound is smaller than the one in \re{sum-of-two} for $q\geq 4$. Finally, when 
$(|P_1|,|P_2|)=(5,5)$, by \rt{5plus5}, $(P_1,P_2)$ is equivalent to $(K_1,K_1)$ and, hence, 
$N_{f_1}+N_{f_2}\leq 2(q-1)^2+4q+4$, which is less than  the bound in \re{sum-of-two} for $q\geq 5$.

Lastly, for (4) we use \rc{4-and-more}. By  \rp{special}, the 
largest bound is obtained in case (iv), i.e., when exactly one of the $P_i$ is equivalent to $E$ and the rest are
equivalent to~$S_2$. In this case we have $N_f\leq L\,(q-1)^2+2k+1\leq L\,(q-1)^2+2L+1$, as stated.
\end{pf}

Note that the assumption  $\cchar(\F_q)>41$ in \rt{max-a} is only needed in case (2b). Even in this case it can be relaxed 
depending on which of the 108 finite classes the corresponding Newton polytope belongs to, see \rr{char41}.

We remark that the bounds in \rt{max-a} are rarely sharp, as the factors often have common zeros which we do not take into account in the proof. For an instance when the bound $N_f\leq L\,(q-1)^2$ is sharp, see \rp{simplex} below.

\begin{Ex}\label{Ex:T0+T0} 
Let $P_1=[2e_1+e_2,e_1+2e_2,0]$, $P_2=[3e_1, e_3,2e_3]$, and $P=P_1+P_2$. Note that $P_1$ and $P_2$
are equivalent to $T_0$ and $L(P)=2$, according to \rt{T_0+four}. Let $\F_7$ be the field of size 7 and consider
$f_1=x^2y-2xy^2+1$ and $f_2=x^3-2z+z^2$ in $\F_7[x,y,z]$. The $\F_7$-space $\cL_P$ has dimension $|P|=15$. The polynomial $f=f_1f_2$ is an element of $\cL_P$ with the largest possible number of non-unit factors. One can check that 
each of $f_1$ and $f_2$ has $54$ zeros in $(\F_7^*)^3$. Moreover, they have 12 zeros in common and, hence, $N_f=96$. 
In fact, this is the largest number of zeros over all non-zero $f\in\cL_P$, as confirmed using {\sc Magma}.
Let us compare this to the bounds we obtained above. For $q=7$ the bound for $N_{f_i}$ in the fourth case of \rp{special} equals $60$; and the bound for $N_f$ in case (3) of \rt{max-a} equals $120$.

\end{Ex}

As we saw in \rt{max-a}, the difference 
$N_f-L\,(q-1)^2$ can have various orders of magnitude depending on the number of summands in the maximal decomposition with 4 or more lattice points (i.e. the number of absolutely irreducible factors of $f$ with 4 or more monomials).  In our next result we give a universal bound for $N_f$ in terms of the Minkowski length only, when $q$ is sufficiently large. The threshold for $q$ also only uses the Minkowski length and the volume of $P$ and does not require any knowledge of maximal decompositions in $P$. The corresponding result for bivariate polynomials is contained
in \cite[Theorem 2.5]{SoSo1}.

\begin{Cor}\label{C:max}
Let $P\subset [0,q-2]^3$ be a lattice polytope and $L=L(P)$ its Minkowski length. 
Assume $\cchar(\F_q)>41$ and $q\geq (c+\sqrt{c^2+1})^2$, where $c=\frac{1}{8}\left(\Vol_3(P)-3L+3\right)$. Consider $f\in\cL_P$ with the largest number of absolutely irreducible factors. Then
$$N_f\leq L\,(q-1)^2+2(q-1)(\floor{2\sqrt{q}}-1).$$
\end{Cor}

\begin{pf} This follows from the observation that, when 
$q$ satisfies the above condition, the bound in case (3) of \rt{max-a} is the maximum of all the bounds listed in there. Clearly, it is larger than the bounds in (1) and (2a). Also, since $P\subset [0,q-2]^3$, we have $L=L(P)\leq L([0,q-2]^3)=3(q-2)$, see \rs{mink-length}. Since 
$6(q-2)+1\leq 2(q-1)(\floor{2\sqrt{q}}-1)$, it follows that the bound in (4) is less than the one in (3). Finally, for the bound in (2b), the inequality
$$(\Vol_3(P)-3L+1)q+2\leq 2(q-1)(\floor{2\sqrt{q}}-1)$$
holds if and only if $q\geq (c+\sqrt{c^2+1})^2$, where $c=\frac{1}{8}\left(\Vol_3(P)-3L+3\right)$, by a direct calculation.
\end{pf}

In the next theorem we show that for large enough $q$ the bound in \rc{max} holds for all non-zero $f\in\cL_P$.
We let  $d$ be the smallest integer such that $P$ is contained in $d\Delta^3$ up to a lattice translation.  
We have $L=L(P)\leq L(d\Delta^3)=d$, see \rs{mink-length}. The case $L=1$ was considered in the previous section, so we will assume that $L\geq 2$. Also, the case $d=L=2$ is simple: $N_P\leq 2(q-1)^2$, as follows from \rp{simplex} below. Thus, we will also assume that $d\geq 3$.

Define $\alpha(P)$ to be the smallest value of $q\geq 7507$ satisfying
the following two inequalities: 

\begin{equation}\label{e:alpha}
\begin{cases}
\ (L-1)q^2-(d^2-3d-2)q^{3/2}-\left(12(d+3)^4+2(L+1)\right)q-4q^{1/2}+L+2\geq 0,\\
\ q^2-(d_1^2-3d_1-2)q^{3/2}-\left(12(d_1+3)^4+6\right)q-4q^{1/2}+4\geq 0,\\   
 \end{cases}
 \end{equation}
where $d_1=d-L+2$. 
Note that when $L=2$ the two inequalities coincide. 

\begin{Th}\label{T:all} 
Let $P\subset d\D^3$ be a lattice polytope of Minkowski length $L$ and $3\leq d< q$.
Let $\alpha(P)$ be as above. Then for any $q\geq \alpha(P)$
we have the following bound for the maximal number of $\F_q$-zeros over all $0\neq f\in\cL_P$
\begin{equation}\label{e:bound}
N_P\leq L\,(q-1)^2+2(q-1)({2\sqrt{q}}-1).
\end{equation}
\end{Th}

\begin{pf} 
Consider a non-zero polynomial $f\in\cL_P$. Let $f=f_1\cdots f_m$ be the factorization into absolutely irreducible factors.
If $m=L$ then the statement follows from \rc{max}, so we assume that $1\leq m\leq L-1$. Let $k$ be the number of the $f_i$
of degree greater than one. If $k=0$ then $N_f\leq m(q-1)^2$ which clearly is less than the bound in \re{bound}, hence, we assume that $k\geq 1$. Without loss of generality we may assume that $f_1,\dots, f_k$ have degrees greater than one.

Let $d_i=\deg f_i$ and let $d'=\deg f$. We have $d'=d_1+\dots+d_m=d_1+\dots+d_k+(m-k)$ and, hence, $d'\geq m+k$. According to \re{affine}, for $1\leq i\leq k$,
\begin{equation}\label{e:crude}
N_{f_i}\leq q^2+(d_i-1)(d_i-2)q^{3/2}+12(d_i+3)^4q,
\end{equation}
and, hence,
$$N_f\leq (m-k)(q-1)^2+kq^2+\sum_{i=1}^k(d_i-1)(d_i-2)q^{3/2}+12(d_i+3)^4q.$$
Note that the function $\phi(d_1,\dots,d_k)=\sum_{i=1}^k(d_i-1)(d_i-2)q^{3/2}+12(d_i+3)^4q$ is convex and, hence, its maximum on the simplex $d_1+\dots+d_k=d'-m+k$, $d_i\geq 2$ is attained at the vertices. Therefore,
\begin{equation*}
\begin{split}
\phi(d_1,d_2,\dots,d_k)&\leq\phi(d'-m-k+2,2,\dots,2)\\
&=(d'-m-k+1)(d'-m-k)q^{3/2}+12\left((d'-m-k+5)^4+ 5^4(k-1)\right)q.
\end{split}
\end{equation*}
Note that the right hand side of the above inequality is increasing in $d'$.
As $d'\leq d$, this implies that $N_f$ is bounded above by the function
 \begin{equation*}
\begin{split}
\psi(k,m)&=(m-k)(q-1)^2+kq^2+(d-m-k+1)(d-m-k)q^{3/2}\\
&+12\left((d-m-k+5)^4+ 5^4(k-1)\right)q.
 \end{split}
\end{equation*}

The function $\psi(k,m)$ is defined on the polygon $\Pi$ given by $1\leq k\leq m\leq L-1, k+m\leq d$. Direct calculations 
show that  $\psi(k,m)$ is convex on $\Pi$ and attains maximum at one of its vertices. If
$d\geq 2(L-1)$ then $\Pi$ is a triangle with vertices $\{(1,1),(1,L-1),(L-1,L-1)\}$. In this case 
$\psi(1,L-1)\geq\psi(L-1,L-1)$ for all values of $q>0$. 
Indeed, when $L=2$ they are equal and when $L\geq 3$ we have
\begin{equation*}
\begin{split}
&\frac{\psi(1,L-1)-\psi(L-1,L-1)}{L-2}=\\
&(2d-3L+3)q^{3/2}+\left(12(2d-3L+12)\left((d-L+5)^2+(d-2L+7)^2\right)-12\cdot 5^4-2\right)q+1\geq\\
&2q^{3/2}+550q+1>0,
\end{split}
\end{equation*}
where in the second to last inequality we used $d\geq 2(L-1)$ and $L\geq 3$.
Therefore, 
$$N_f\leq\psi(k,m)\leq \max\{\psi(1,1),\psi(1,L-1)\}.$$

If $d< 2(L-1)$ then $\Pi$ has four vertices $\{(1,1),(1,L-1),(d-L+1,L-1),(d/2,d/2)\}$.
First, we see that $\psi(1,L-1)\geq \psi(d-L+1,L-1)$. Indeed, 
\begin{equation*}
\begin{split}
\psi(1,&L-1)-\psi(d-L+1,L-1)=\\
&(d-L)(d-L+1)q^{3/2}+\left(12(d-L+5)^4-12\cdot 5^4(d-L+1)-2(d-L)\right)q+d-L.
\end{split}
\end{equation*}
One directly checks that each coefficient above is zero when $d=L$ and positive when $d-L\geq 1$.

Next we note that for $q\geq 7507$ we also have $\psi(d-L+1,L-1)>\psi(d/2,d/2)$. Indeed,
\begin{equation*}
\begin{split}
\psi(d-L+1,L-1)-\psi(d/2,d/2)=\frac{1}{2}\left(2(L-1)-d\right)\left(q^2-7504q+2\right)>0.
\end{split}
\end{equation*}
It remains to note that $\psi(1,1)\leq L\,(q-1)^2+2(q-1)({2\sqrt{q}}-1)$ is equivalent to the first inequality
in \re{alpha} and $\psi(1,L-1)\leq L\,(q-1)^2+2(q-1)({2\sqrt{q}}-1)$ is equivalent to the second inequality in \re{alpha}.
Therefore, 
$$N_f\leq \max\{\psi(1,1),\psi(1,L-1)\}\leq L\,(q-1)^2+2(q-1)(\floor{2\sqrt{q}}-1),$$
for all $q\geq\alpha(P)$.
 \end{pf}

To conclude this section, we mention one case when we obtain a sharp upper bound on $N_P$ for all values of $q\geq 2$.
As we mentioned in the introduction, when $P$ is the $L$-dilate of the standard simplex, $P=L\D^n$, 
we have  $N_P= L\,(q-1)^{n-1}$ by considering
the product of $L$ linear factors in one variable. The next proposition is a slight generalization of this fact. The method used in the proof is
standard and is a part of folklore. 

\begin{Prop}\label{P:simplex} 
Let $P$ be a lattice polytope  
equivalent to a subpolytope of  $L\D^n$. Then  $N_P\leq L\,(q-1)^{n-1}$. Moreover, $N_P= L\,(q-1)^{n-1}$ if and only if
$P$ contains a lattice segment of lattice length $L$. 
\end{Prop}

\begin{pf} 
We prove the first statement and the ``only if" part of the second statement by induction on $n$. The case $n=1$ is trivial, so assume $n>1$.
By applying an affine unimodular map to $P$ we may assume $P\subseteq L\D^n$ and hence
every non-zero $f\in\cL_P$ has degree at most $L$. Thus, we can write
$$f=\sum_{i=0}^m f_ix_n^i,$$
where the $f_i$ are polynomials in $x_1,\dots,x_{n-1}$ and $m$ is the largest integer such that $f_m$ is non-zero. Note that $\deg f_m\leq L-m$.
Consider
$$Z=\{z\in(\F_q^*)^{n-1} : f_0(z)=\dots=f_m(z)=0\}.$$
If $z\in Z$ then $f(z,x_n)=0$ for any $x_n\in\F_q^*$. If $z\not\in Z$ then $f(z,x_n)$ is a non-zero polynomial in $x_n$ of degree
at most $m$ and, hence, has at most $m$ zeros. Therefore,
\begin{equation}\label{e:Z}
N_f\leq |Z|(q-1)+\left((q-1)^{n-1}-|Z|\right)m=m(q-1)^{n-1}+|Z|(q-1-m).
\end{equation}
On the other hand, $|Z|\leq N_{f_m}\leq (L-m)(q-1)^{n-2}$, where the last inequality follows by the inductive hypothesis applied to $P=(L-m)\D^{n-1}$. Combining this with \re{Z} produces
\begin{equation}\label{e:first}
N_f\leq L\,(q-1)^{n-1}-m(L-m)(q-1)^{n-2}\leq L\,(q-1)^{n-1},
\end{equation}
and the first statement follows.

Now let $f\in\cL_P$ be such that $N_f=L\,(q-1)^{n-1}$. Then in \re{first} we have equality and, hence, either $m=0$ or $m=L$.
In the first case $f=f_0$ whose support is contained in $P\cap L\D^{n-1}$. By induction, this implies that $P\cap L\D^{n-1}$
(and, hence, $P$) contains a lattice segment of length $L$. In the second case we see that the point $Le_n$ belongs to $P$.
Replacing $x_n$ with $x_{n-1}$ in the above argument we again see that  $P$ contains either a lattice segment of length $L$
or the point $Le_{n-1}$. But if it contains both $Le_{n-1}$ and $Le_n$ then it contains the segment $[Le_{n-1},Le_n]$, by the convexity of $P$.

Finally, the``if" part of the second statement is clear since after an affine unimodular map we may assume that $P$ contains the
segment  $[0,Le_1]$. Then, for distinct $\alpha_1,\dots, \alpha_L\in\F_q^*$, the polynomial $f=(x_1-\alpha_1)\cdots(x_1-\alpha_L)$ 
is contained in $\cL_P$ and has exactly $L\,(q-1)^{n-1}$ zeros in $(\F_q^*)^n$.
\end{pf}

\section{Number of $\F_q$-zeros of polynomials in $\cL_P$ when $P$ has lattice width one}

As we have seen in the previous section, the values of the function $\alpha(P)$ in \rt{all}  are quite large. We expect that the bound  \re{bound} holds for smaller values of $q$ as well, but the precise lower bound for $q$ depends on the geometry of $P$ in a non-trivial way. The reason why our threshold is so large is because we used rather crude estimate  \re{crude} that involve the degree of the $f_i$ and do not take into account the combinatorics 
of their Newton polytopes. To the best of our knowledge, there are no known generalizations
of \re{crude} (and, more generally, of \re{affine}) that estimate the number of 
$\F_q$-zeros of absolutely irreducible polynomials in terms of the Newton polytope, rather than the degree,
besides the case $L(P)=1$ in \cite[Th 4.30]{Josh} (see \rt{JoshFinite}) and when
$P$ has lattice width one (see \rp{Infinite}).
In this section we use \rp{Infinite} to produce a better bound for $N_P$ than \re{bound} for $P$ of lattice width one,
which holds for all $q$ starting with a much smaller threshold.


Let $P\subset\R^3$ be a lattice polytope of lattice width one. After a unimodular transformation we may assume 
that $P=\conv(P_0\times\{0\})\cup (P_1\times\{1\})$ for some lattice polytopes $P_0,P_1$ in $\R^2$. 
As in \rl{mix}, let $V(P_0,P_1)=\Vol_3(P)-\Vol_2(P_0)-\Vol_2(P_1)$ be the mixed volume of $P_0$ and $P_1$.
Define $\beta(P)$ to be
\begin{equation}\label{e:beta}
\beta(P)=\max\{37, (C+\sqrt{C^2+5/2})^2, (c+\sqrt{c^2+3})^2, (V(P_0,P_1)+1)^2/4\}, 
\end{equation}
where $C=\max_{i=0,1}\{\Vol_2(P_i)/4-L(P_i)+9/4\}$, and $c=\min_{i=0,1}\{\Vol_2(P_i)/2-2L+11/2\}$.
\begin{Th}\label{T:width-one} 
Let $P=\conv(P_0\times\{0\})\cup (P_1\times\{1\})$ be a lattice polytope of width one and Minkowski length $L$, and let
$\beta(P)$ be as above. Then for any $q\geq \beta(P)$
we have the following bound for the maximal number of $\F_q$-zeros over all $0\neq f\in\cL_P$
\begin{equation}
N_P\leq L\,(q-1)^2+(q-1)(\lfloor 2\sqrt{q}\rfloor-1).
\end{equation}
\end{Th}

\begin{pf} In the proof we make extensive use of the bounds proved in \cite{SoSo1}. The reader should 
keep in mind that even when a Laurent polynomial depends on $x,y$ only, we are still interested
in its zeros in $(\F_q^*)^3$. This is why the bounds from \cite{SoSo1}  are  multiplied by a factor of $(q-1)$
in the arguments below.

Consider $0\neq f\in\cL_P$. First, assume that $f\in\F_q[x^{\pm 1},y^{\pm 1}]$. Then, 
applying to $P_0$ Theorem 2.6, part (1) in \cite{SoSo1},  we obtain 
$$N_f\leq L\,(q-1)^2+(q-1)(\lfloor 2\sqrt{q}\rfloor-1)$$
for all $q\geq\beta(P)$. Here we used $L(P_0)\leq L$ and the fact that $\beta(P)$ is larger than the threshold that appears 
in \cite[Th 2.6 (1)]{SoSo1}. The case when $f=hz$ for some $h\in\F_q[x^{\pm 1},y^{\pm 1}]$ is analogous, with
$P_0$ replaced by $P_1$.

It remains to consider the case when $f=f_0+f_1z$ for non-zero $f_0,f_1\in\F_q[x^{\pm 1},y^{\pm 1}]$. We write
$f=gh$ where $g=g_0+g_1z$ is absolutely irreducible and $g_0,g_1,h\in\F_q[x^{\pm 1},y^{\pm 1}]$ are non-zero. 
As the Newton polytope $P_g$ is of positive dimension, we have
$$L(P)\geq L(P_f)=L(P_g+P_h)\geq L(P_g)+L(P_h)\geq 1+L(P_h)$$ 
and, hence, $L(P_h)\leq L-1$. We have two cases: (a) $L(P_h)= L-1$ and (b) $L(P_h)\leq L-2$.

(a) Assume $L(P_h)= L-1$ and, hence, $L(P_g)=1$. Suppose there is a maximal decomposition in $P_h$ which contains 
a summand with 4 lattice points (i.e. is equivalent to $T_0$). By adding $P_g$ we obtain a maximal decomposition
in $P$. Then \rc{width-one} implies that $P_g$ has at most 3 lattice points and, hence, $N_g\leq (q-1)^2$, see \rp{special}.
On the other hand, applying \cite[Th 2.6 (1)]{SoSo1} to $P_h$, we obtain
$$N_h\leq (L-1)\,(q-1)^2+(q-1)(\lfloor 2\sqrt{q}\rfloor-1)$$
for all $q\geq\beta(P)$. Here $L(P_h)=L-1$ and $\beta(P)$ is greater than the threshold in \cite[Th 2.6 (1)]{SoSo1}
since $P_{g_ih}\subseteq P_i$ and so the area of $P_h$ cannot exceed the area of each $P_i$, for $i=0,1$.
Combining the above two inequalities, we get
$$N_f\leq N_g+N_h\leq L\,(q-1)^2+(q-1)(\lfloor 2\sqrt{q}\rfloor-1),$$
for $q\geq\beta(P)$.

Now suppose there is no maximal decomposition in $P_h$ which contains 
a summand with 4 lattice points. Then by \cite[Th 2.6 (2)]{SoSo1} applied to $P_h$ we have
$N_h\leq (L-1)(q-1)^2$ for $q\geq \beta(P)$. 
To bound $N_g$ we use
\rp{Infinite} (note that $g_0$ and $g_1$ do not have common factors, as $g$ is irreducible):
$$N_g\leq (q-1)^2+V(P_{g_0},P_{g_1})q.$$
By picking a point $v\in P_h$ we see that $P_{g_i}+v\subseteq P_{g_ih}\subseteq P_i$, for $i=0,1$. Thus,
by monotonicity and invariance of the mixed volume (see \rs{mix-vol}) we have
 $$V(P_{g_0},P_{g_1})=V(P_{g_0}+v,P_{g_1}+v)\leq V(P_0,P_1).$$
Combining, we obtain
$$
N_f\leq N_g+N_h\leq L\,(q-1)^2+V(P_0,P_1)q.
$$
Since $\beta(P)\geq (V(P_0,P_1)+1)^2/4$, we see that the
right hand side of the above inequality is no greater than $L\,(q-1)^2+(q-1)(\lfloor 2\sqrt{q}\rfloor-1)$ for $q\geq\beta(P)$,
as required.

(b) Assume $L(P_h)\leq L-2$. We mimic the proof of  \cite[Th 2.6]{SoSo1}
to show 
\begin{equation}\label{e:N_h}
N_h\leq (L-1)(q-1)^2,\quad \text{for }\ q\geq \beta(P).
\end{equation} 
Then, repeating the argument at the end of case (a), we get the required bound.

Let $h=h_1\cdots h_k$ be a factorization into absolutely irreducible factors. Note
that $k\leq L-2$. First assume that none of the $P_{h_i}$ has interior lattice points. Then 
each $P_{h_i}$ either has lattice width one (as a polytope in $\R^2$) or is equivalent to $2\D^2$.
In the former case $N_{h_i}\leq (q-1)^2$ and in the latter case $N_{h_i}\leq (q+1)(q-1)$, see \cite[Cor 2.2]{SoSo1}.
Let $s$ be the number of $P_{h_i}$ equivalent to $2\D^2$. We have
$$L-2\geq L(P_h)\geq \sum_{i=1}^kL(P_{h_i})\geq 2s+(k-s)=k+s.$$
Then
\begin{align*}
N_h\leq \sum_{i=1}^kN_{h_i}&\leq s(q+1)(q-1)+(k-s)(q-1)^2=\left(2(k+s)+k(q-3)\right)(q-1)\\
&\leq \left(2(L-2)+(L-2)(q-3)\right)(q-1)=(L-2)(q-1)^2.
\end{align*}

Now assume that at least one $P_{h_i}$ has interior lattice points. Then, by \cite[Lem 2.7]{SoSo1}, we have
$$N_h\leq k(q-1)^2+(B\sqrt{q}+2)(q-1),$$
where $B=\Vol_2(P_h)-4k+3$. As before, the area of $P_h$ is no greater than the area of $P_i$, for $i=0,1$, so
we may replace $\Vol_2(P_h)$ in $B$ with the smaller of the $\Vol_2(P_i)$. To ensure that 
$$k(q-1)^2+(B\sqrt{q}+2)(q-1)\leq (L-1)(q-1)^2$$
we must choose $q$ large enough to satisfy
$$(L-k-1)q-B\sqrt{q}-(L-k+1)\geq 0,$$
which is a quadratic inequality in $\sqrt{q}$. Arguments similar to the ones in the proof of Theorem 2.6, part (1) in \cite{SoSo1} show that
it suffices to choose $q\geq \max\{37, (c+\sqrt{c^2+3})^2\}$, where 
$c=\min_{i=0,1}\{\Vol_2(P_i)/2-2L+11/2\}$.

\end{pf}

\begin{Ex}\label{Ex:width-one}
Let $P=T_0+D$, where $D=[0,e_1,e_3]$. This is a polytope of lattice width one as in \rt{width-one}, with
$P_0=T_0+I$ and $P_1=T_0$, where $I=[0,e_1]$, see \rf{width-one}. 
\begin{figure}[h]
\begin{center}
\includegraphics[scale=.35]{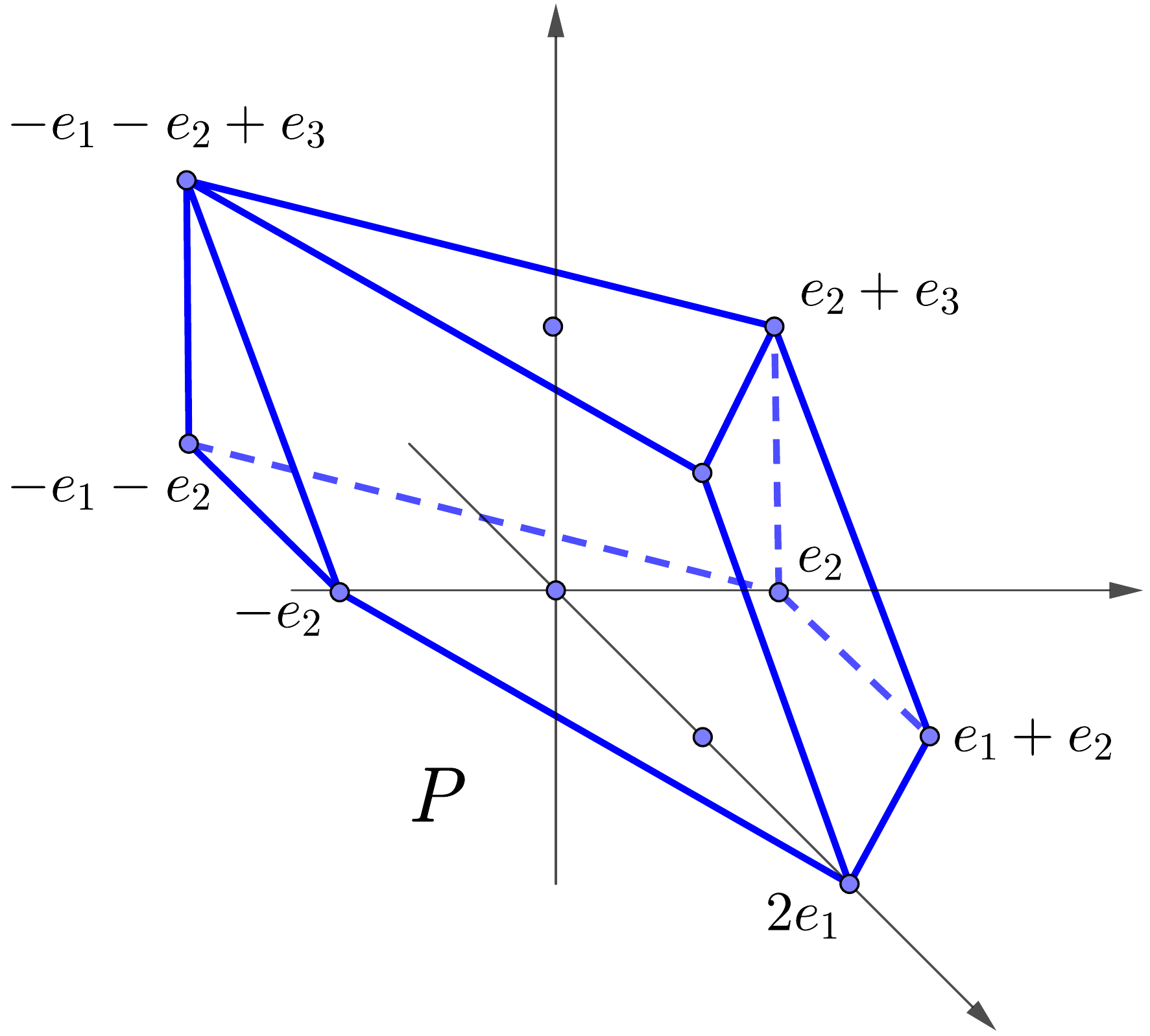}
\end{center}
\caption{Lattice polytope of width one with $L(P)=2$}
\label{F:width-one}
\end{figure}
It is easy to check that $L(P)=L(P_0)=2$ and
$L(P_1)=1$. Also, $\Vol_3(P)=15$, $\Vol_2(P_0)=7$, $\Vol_2(P_1)=3$, and so $V(P_0,P_1)=5$.
\rt{width-one} provides 
\begin{equation}\label{e:bound-ex}
N_P\leq 2(q-1)^2+(q-1)(\lfloor 2\sqrt{q}\rfloor-1)
\end{equation}
for all $q\geq 107$.
In fact, this bound holds for all $q\geq 5$. Indeed, since $L(P)=2$, any non-zero $f\in\cL_P$ is either absolutely
irreducible or is a product of two absolutely irreducible factors. In the first case
$N_f\leq (q-1)^2+5q$, by \rp{Infinite} and monotonicity of the mixed volume. In the second case,
at most one of the factors, say $f_1$, has 4 or more monomials. Then $N_{f_2}\leq (q-1)^2$.
If $P_{f_1}$ is equivalent to $T_0$ then $N_{f_1}\leq (q-1)^2+(q-1)(\lfloor 2\sqrt{q}\rfloor-1)$
and so
$$N_f\leq N_{f_1}+N_{f_2}\leq 2(q-1)^2+(q-1)(\lfloor 2\sqrt{q}\rfloor-1)$$
for all $q\geq 2$. Otherwise, $P_{f_1}$ has width one. Note that the top face of $P_{f_1}$ must be a
vertex, otherwise $P_1$ would contain a sum of two segments which contradicts $L(P_1)=1$.
But then the mixed volume of the top and the bottom faces of $P_{f_1}$ is zero and
 \rp{Infinite} produces $N_{f_1}\leq (q-1)^2$. Thus, 
 $$N_f\leq N_{f_1}+N_{f_2}\leq 2(q-1)^2.$$
It remains to notice that $(q-1)^2+5q\leq 2(q-1)^2+(q-1)(\lfloor 2\sqrt{q}\rfloor-1)$ for all $q\geq 5$.
We used {\sc Magma} to compare this bound with the actual values of $N_P$ for  $5\leq q\leq 11$:
\renewcommand{\arraystretch}{1.15}
\begin{center}
\begin{tabular}{|r||c|c|c|c|c|}
\hline
$q\ \ $ & $\ \ 5\ \ $ & $\ \ 7\ \ $ & $\ \ 8\ \ $ & $\ \ 9\ \ $ & $\ 11\ $  \\
\hline
$N_P\ $ & 40 & 90 & 112 & 160 & 250\\
\hline
{\rm bound in \re{bound-ex}} & 44 & 96 & 126 &168 & 250\\
\hline
\end{tabular}
\end{center}
In particular, this shows that the bound in \re{bound-ex} is sharp.
\end{Ex}
\medskip

\section{Examples of toric 3-fold codes of dimension 8 over small fields}\label{S:examples}

In this section we consider toric 3-fold codes defined by polytopes with $L(P)=1$ which have 8 lattice points and lattice width greater than one.  
Recall that a linear code $\cC$ is an $[n,k,d]_q$-code if it is defined over $\F_q$, has length $n$, dimension $k$, and minimum distance $d$. For a toric 3-fold code
$\cC_P$ we have $n=(q-1)^3$, $k=|P|$, and $d=n-N_P$, see \rs{motivation}.

Since our bound in \rc{toric-app} depends on $L(P)$ and not $k=|P|$, it is reasonable to look for good toric codes among the ones whose polytope $P$ has the largest 
possible value of $|P|$ for fixed $L(P)$. By \re{eight}, $|P|=8$ is the largest possible number of lattice points for 3-dimensional polytopes with $L(P)=1$.
According to the classification in \cite{BlancoSantos3,Josh},
there are exactly five lattice polytopes with $L(P)=1$, $|P|=8$, and width greater than one, up to $\AGL(3,\Z)$-equivalence.
Using {\sc Magma}, we computed the parameters of all five toric $3$-fold codes over $5\leq q\leq 13$.
For $q=5,7,8,11,13$ the best parameters are produced by a clean lattice simplex with four interior lattice points,
$P=\left[\begin{matrix}0 & 1 & 0 & 6  \\ 0 & 0 & 1 & 8 \\ 0 & 0 & 0 & 35\end{matrix}\right]$.
For $q=9$ the best parameters are produced by a 5-vertex polytope with three interior lattice points, 
$Q=\left[\begin{matrix}0 & 1 & 0 & 0 & 2  \\ 0 & 0 & 0 & 1 & 15 \\ 0 & 1 & 1 & 1 & 28\end{matrix}\right]$.
Although $P$ and $Q$ do not lie in $[0,q-2]^3$ for $5\leq q\leq 13$, one can check that their lattice points
are distinct modulo $(\Z/(q-1)\Z)^3$.

One feature of  toric 3-fold codes is that the length $n$ grows very quickly and already starting with $q=7$ 
their parameters  fall outside of the ranges set in linear code tables such as MinT and Grassl's table \cite{Gra,MinT}. 
Thus, we compare the parameters of $\cC_P$ and $\cC_Q$ with the Griesmer and the Gilbert-Varshamov bounds.
Recall that the Griesmer bound states that for any linear $[n,k,d]_q$-code
$n\geq \sum_{i=0}^{k-1}\lceil d/q^i\rceil$ and, hence, provides
an upper bound on $d$ when $k$ and $n$ are fixed, \cite[Th 1.1.43]{TVZ}. The Gilbert-Varshamov bound asserts that if
\begin{equation}\label{e:GV}
q^{n-k}>\sum_{i=0}^{d-2}{n-1\choose i}(q-1)^i
\end{equation}
then there exists a linear $[n,k,d]_q$-code \cite[Th 1.1.59]{TVZ}. Therefore, it is desirable to find linear codes whose minimum distance exceeds the largest
value of $d$ which satisfies \re{GV} with fixed $n$ and $k$. 
We have the following table of values of $d$ for $\cC_P$ and $\cC_Q$ for $5\leq q\leq 13$.
\renewcommand{\arraystretch}{1.15}
\begin{center}
\begin{tabular}{|r||c|c|c|c|c|c|}
\hline
$q\ \ $ & $\ \ 5\ \ $ & $\ \ 7\ \ $ & $\ \ 8\ \ $ & $\ \ 9\ \ $ & $\ 11\ $ & $\ 13\ $ \\
\hline
\hline
{\rm Griesmer upper bound} & 47 & 181 & 296 & 451 & 904 & 1590\\
\hline
$d(\cC_P)\ $ & 36 & \textcolor{red}{162} & 252 & 392 & \textcolor{red}{861} & \textcolor{red}{1535}\\
\hline
$d(\cC_Q)\ $ & 36 & 150 & 252 & \textcolor{purple}{416} & 850 & 1512\\
\hline
{\rm Glibert-Varshamov lower bound} & 37 & 159 & 268 &416 & 857 & 1519\\
\hline
\end{tabular}
\end{center}
\medskip
Although for $q=5$ the resulting codes have minimum distance far from the best known $[64,8,42]_5$-code in \cite{Gra}, for
$q=9$ the minimum distance of $\cC_Q$ (in purple) meets the Gilbert-Varshamov bound and for $q=7, 11, 13$ the minimum
distance of $\cC_P$ (in red) exceeds the Gilbert-Varshamov bound.

\end{document}